\documentclass[reqno,11pt]{amsart}
\usepackage{geometry}
\usepackage{amsmath,amssymb,mathrsfs,color,float,paralist}
\usepackage[colorlinks,
linkcolor=red,
anchorcolor=green,
citecolor=blue
]{hyperref}
\usepackage{lipsum}
\newcommand*\justify{%
  \fontdimen2\font=0.4em% interword space
  \fontdimen3\font=0.2em% interword stretch
  \fontdimen4\font=0.1em% interword shrink
  \fontdimen7\font=0.1em% extra space
  \hyphenchar\font=`\-% allowing hyphenation
}% solve line break problem in texttt

\usepackage{fourier}
\usepackage{calc}
\newtheorem{The}{Theorem}[section]
\newtheorem{Cor}[The]{Corollary}
\newtheorem{Lem}[The]{Lemma}
\newtheorem{Pro}[The]{Proposition}

\theoremstyle{remark}
\newtheorem{Rem}{Remark}
\numberwithin{equation}{section}

\newcommand{\T}{\mathbb{T}}
\newcommand{\R}{\mathbb{R}}
\newcommand{\Z}{\mathbb{Z}}

\newcommand{\CC}{\mathbb{C}}

\newcommand{\cS}{\mathcal{S}}
\newcommand{\cL}{\mathcal{L}}
\newcommand{\cU}{\mathcal{U}}
\newcommand{\cV}{\mathcal{V}}
\newcommand{\cW}{\mathcal{W}}
\newcommand{\cB}{\mathcal{B}}
\newcommand{\cD}{\mathcal{D}}
\newcommand{\cM}{\mathcal{M}}
\newcommand{\HH}{\mathcal{H}}

\newcommand{\La}{\Lambda}
\newcommand{\tcM}{\widetilde{\mathcal{M}}}
\newcommand{\tLa}{\widetilde{\Lambda}}
\newcommand{\tPh}{\widetilde{\Phi}}
\newcommand{\tGa}{\widetilde{\Gamma}}
\newcommand{\tsig}{\widetilde{\sigma}}
\newcommand{\tk}{\widetilde{k}}
\newcommand{\vep}{\varepsilon}
\newcommand{\vph}{\varphi}
\newcommand{\wtH}{\widetilde{\mathcal{H}}}
\newcommand{\whf}{\widehat{f}}
\newcommand{\wts}{\widetilde{s}}

\title{Analytic genericity of diffusing orbits in \textit{a priori} unstable Hamiltonian systems}

\author{Qinbo Chen}

\address{Department of Mathematics, KTH Royal Institute of Technology, 10044 Stockholm, Sweden}
\email{qinbochen1990@gmail.com}

\author{Rafael de la Llave}
\address{School of Mathematics, Georgia Institute of Technology, Atlanta, GA 30332, USA}
\email{rafael.delallave@math.gatech.edu}


\begin{document}
\begin{abstract}
The genericity of Arnold diffusion in the analytic category is an open problem. In this paper, we  study  this problem  in the following \textit{a priori} unstable Hamiltonian system with a time-periodic perturbation
\[\mathcal{H}_\varepsilon(p,q,I,\varphi,t)=h(I)+\sum_{i=1}^n\pm \left(\frac{1}{2}p_i^2+V_i(q_i)\right)+\varepsilon H_1(p,q,I,\varphi, t), \]
where $(p,q)\in \mathbb{R}^n\times\mathbb{T}^n$, $(I,\varphi)\in\mathbb{R}^d\times\mathbb{T}^d$  with $n, d\geq 1$,   $V_i$ are Morse potentials, and $\varepsilon$ is a  small non-zero parameter. The unperturbed Hamiltonian is not necessarily convex,  and the induced inner dynamics  does not need to satisfy a twist condition. Using geometric methods we prove that Arnold diffusion occurs for generic analytic perturbations $H_1$. Indeed, the set of admissible $H_1$ is $C^\omega$ dense and $C^3$ open (a fortiori,  $C^\omega$ open). Our perturbative technique for the genericity is valid in the $C^k$ topology for all  $k\in [3,\infty)\cup\{\infty, \omega\}$.
\end{abstract}

\maketitle

\section{Introduction}\label{section_Introduction}
The goal of this paper is to study the Arnold diffusion problem for analytic perturbations of  a given
\textit{a priori} unstable Hamiltonian system.
Arnold diffusion is a phenomenon of instability in Hamiltonian systems with more than two degrees of freedom. 
This problem arises in the study of the effect of small perturbations on integrable systems, and has attracted a lot of attention both in mathematics and in physics due to its importance for the applications. 

For a nearly integrable Hamiltonian system,
the celebrated Kolmogorov-Arnold-Moser (KAM)  theory asserts that the most part  (in the measure-theoretic sense) of the phase space is filled with KAM invariant tori carrying quasi-periodic dynamics. Arnold diffusion asks for the large scale motions in the complement of KAM tori.
The first example is constructed by V. I. Arnold in \cite{Ar1964}. He also conjectured that the diffusive phenomena occur for generic  systems: {\small \texttt{\justify  the typical case in a multidimensional  problem is topological instability: through an arbitrarily small neighborhood of any point there passes a  phase trajectory along which the action variables go away from the initial values by a quantity of order one}} \cite[Chapter 6]{Arnoldbook2006}.  One of the main problems in this conjecture  is the genericity  in some appropriate function space (e.g. $C^r$-differentiable, analytic).  The genericity in the $C^r$-differentiable topology is now well understood. However, as pointed out in a recent survey \cite{Cheng_Xue2019}, it remains a deep open problem to prove Arnold diffusion in the analytic category. This issue is of great interest since  many Hamiltonians with physical significance are analytic.

In this paper we give an affirmative answer to 
 the $C^\omega$-genericity issue of Arnold diffusion for \textit{a priori} unstable Hamiltonian systems. The \textit{a priori} unstable Hamiltonian system  
 consists of a rotor-pendulum system  plus a time periodic perturbation, and it can  be viewed as a scaled  approximation on the dynamics near simple resonances of the  \textit{a priori} stable systems \cite{CG1994}. 

Note that the proof presented in this paper works even if the functions $V_i(q_i):\T\to \R$, $i=1,\cdots,n$ in the unperturbed part are  small (weak hyperbolicity), so our result applies to some \textit{a priori} stable systems, see Remark \ref{rem_smallV} for more explanation.   
Our approach for the proof follows a recent geometric mechanism established in
  \cite{Gidea_Delallave_seara2020}. 
This mechanism relies on the presence of 
   normally hyperbolic invariant manifold (NHIM), with transverse intersection between the associated stable and unstable manifolds. Indeed, the Melnikov method will be used to show transverse homoclinic orbits in the perturbed system. 
   In this setting, we can then use the theory of scattering maps to compute the effect of homoclinic excursions. Heuristically, 
   the scattering map  gives the future asymptotic of an orbit as a function of its past asymptotic \cite{DLS2000, DLS2008}.  By shadowing the pseudo-orbits of this map,  it allows to  show instability for the original dynamics.
  
  Here, we give a brief overview of the previous works and approaches
  on the genericity problem of Arnold diffusion.  The scattering map
  has become an effective tool to study the phenomena of instability
  in concrete examples or generic systems.  By exploiting this
  geometric tool, Arnold diffusion has been proved to occur for
  generic perturbations in the $C^r$-differentiable topology, see for
  instance
  \cite{DLS2000,DLS2016,Delshams_Huguet2009,Gidea_Delallave_seara2020}.

  In particular, the geometric mechanism developed in
  \cite{Gidea_Delallave_seara2020} requires almost no information of
  the inner dynamics on the NHIM. Only recurrence of the motion in the
  NHIM is needed, and it is automatically satisfied in the Hamiltonian
  case by Poincar\'e recurrence theorem if the motions in the NHIM are
  bounded (Of course, if the motions in the NHIM are not bounded, 
one has diffusion in the NHIM!).

The main
 hypothesis of  the mechanism of \cite{Gidea_Delallave_seara2020}  is 
some explicit transversality conditions, which  are implied 
by checking that some Melnikov-type functions have
non-degenerate critical points.  Sometimes the dynamics of a single
scattering map may have difficulties moving long distances. But if
several scattering maps are available, one can iterate these
scattering maps to find large scale motions. 
For applications of the mechanism of
\cite{Gidea_Delallave_seara2020_skip} in celestial mechanics see e.g.
\cite{CGD_Arnold2017, DKDS_global2019}.

Another mechanism assuming mainly recurrence -- but assuming some 
separation of time scales appears in \cite{Gidea_dlL2017}. 
Geometric methods that  use 
NHIM but assume  that there are some other 
invariant objects in the NHIM (e.g secondary tori) appear in
\cite{DLS2000, DLS2003, DLS2006,Delshams_Huguet2009, DLS2006orbits,DHLS2008,DLS2016}
and applications to celestial mechanics and other concrete models 
appear in \cite{Delshams_Huguet2009,FGKR_2016_Kirkwood, Dels_Scha18}. 
Another important geometric 
method based on  separatrix maps to study  Arnold diffusion
 can be found in \cite{Treschev2002Trajectories,Treschev2004,Treschev2012}, \emph{etc}.  It is worth mentioning that the variational method is also an effective approach to study the diffusion problem. The techniques and ideas developed by J. Mather  \cite{Mather1993, Mather2003} have significant influence. 
Applying global variational methods to convex Hamiltonian systems, several authors have established the genericity  in the $C^r$-differentiable topology, see \cite{Cheng_Yan2004,Bernard2008,Cheng_Yan2009,KZ2015,Cheng2017_DoubleResonace,Cheng2019,Bernard_Kaloshin_Zhang2016,Cheng_Xue2015,KZ_book2020}.  Note, however that in some of
these papers, the notion of genericity is redefined and that, of
course, a convexity condition is needed. 

This paper does not aim to review the rich history on this very active area.
There are many other related works, 
 we  mention here \cite{Bolotin_Treschev1999,Bessi_Chierchia_Valdinoci01,BB2002, Cr2003,LM2005,Gelfreich_Turaev2008,KLS2014, Xue2014,Gidea_dlL2017,GKZ2016,Marco2016Arnold,Chen_Cheng2017,Dels_Scha18,Gelfreich_Turaev2017} and references therein.

 Thus, as mentioned above, the genericity of Arnold diffusion in the $C^r$-differentiable category has been well studied. However, the most difficult case is the analytic genericity, which is still an open problem.
 The difficulty lies in that most of the previous works require the use of non-analytic techniques (e.g. bump functions) for perturbations.

In this paper,
  a new perturbative technique is introduced to solve the analytic genericity of Arnold diffusion.  Following \cite{Gidea_Delallave_seara2020}, the scattering map is used as an essential tool.
  We will take advantage of the Poincar\'e-Melnikov method and the family of periodic potential functions
 to verify the genericity of  some transversality hypotheses. 
As will see below, the novelty of our technique is as follows: (I) \textit{ It is valid in both the $C^r$-differentiable topology  and the $C^\omega$ topology}; (II) \textit{We can also obtain the genericity in the sense of Ma\~n\'e \cite{Mane1996}, that is, the genericity ( actually, $C^\omega$ dense and $C^3$ open ) in the space of   periodic potential functions.}

The basic idea of our method is as follows. The
work of \cite{Gidea_Delallave_seara2020} shows that it suffices
to verify some transversality conditions for the zeros of
a rather explicit Melnikov integral. Following the standard procedure
in transversality, 
we show that, if there are some degenerate
situations, more or less arbitrary perturbations break the degeneracy.
See Section~\ref{section_proof}.

Note that the families of perturbations we choose are rather arbitrary.
Hence, the result is stronger than density. We show that
the transversality -- and hence the diffusion --  can only
fail in an infinite codimension set. See Remark~\ref{rem:codimension}. 
 For practical applications,
we note that the only condition to check is a very explicit
(and rapidly convergent) integral, so that given a concrete
system (e.g. in celestial mechanics), one can verify the result
with a finite precision calculation and obtain quantitative information
on the location of the diffusing orbits.

\section{Setup and Main Result}\label{sec_mainresult}

\subsection{Notation and assumptions}
For positive integers $d\geq 1$ and $ n\geq 1$, we let  
$\cB\subset \R^d$ and $\cD\subset\R^n$ be two open domains with compact closures 
$\overline{\cB}$ and $\overline{\cD}$ respectively. Without loss of generality, we may suppose $\cB=\{ x\in\R^d: \|x\|<R_1 \}$ and $\cD=\{z\in\R^n: \|z\|< R_2\}$, where  $R_1$ and $R_2$ are suitably large, and $\|\cdot\|$ is the standard Euclidean norm.

We consider the following \textit{a priori} unstable system with a time-periodic perturbation:
\begin{equation}\label{main_H_system}
	\mathcal{H}_\vep=H_0(p,q,I)+\vep H_1(p,q,I,\vph, t), \qquad (p, q, I,\vph)\in \cD\times\T^n\times\cB\times\T^d.
	\end{equation}
	Here,  $p=(p_1,\cdots,p_n)$ and $q=(q_1,\cdots,q_n)$ are symplectically conjugate variables,  $I=(I_1,\cdots,I_d)$ and $\vph=(\vph_1,\cdots,\vph_d)$ are  symplectically conjugate variables, and $\T=\R/2\pi\Z$.  The unperturbed Hamiltonian $H_0$ is given by
\begin{equation}\label{unperturbed_H}
H_0(p,q,I)=h(I)+\sum_{i=1}^n \pm\Big[\frac{1}{2}p_i^2+V_i(q_i)\Big].
\end{equation} 
Here, the symbol $\pm$ in  \eqref{unperturbed_H} means that one can take either the plus sign $``+"$ or the minus sign $``-"$ in front of each pendulum $\frac{1}{2}p_i^2+V_i(q_i)$. The whole perturbation term $\vep H_1$ is assumed to be real analytic and  periodic on time $t$ with a period $2\pi$.  The unperturbed part $H_0$ represents a $d$-degree-of-freedom rotator  plus $n$ pendulums. $H_0$ is not necessarily convex, and the induced inner dynamics is not necessarily a twist map.

Throughout this paper, we  assume the following  conditions on $H_0$:
\begin{enumerate}
\item  [\textbf{(H1)}]   $h(I)$ and each $V_i(q_i)$ with $i\in\{1,\cdots, n\}$ are of class $C^{r}$ with the extended
  integer $r\in [3, \infty)\cup\{\infty, \omega\}$.
	\item [\textbf{(H2)}] For each $i$, the function $V_i:\T\to\R $ has a unique maximum point which is non-degenerate in the sense of Morse. Without loss of generality and to simplify the notation, we may always assume the maximum point $q_{_{\textup{max}}}=0$.
\end{enumerate}

Condition {\bf(H2)} tells us that $V'_i(0)=0$ and $V_i''(0)<0$ for each $i\in\{1,\cdots,n\}$. 
It is clear that the non-degeneracy condition {\bf(H2)}    is $C^2$-open and $C^\omega$-dense. We also remark that
the approach used in this paper is also applicable to those  systems whose unperturbed part is $h(I)+\sum_{i=1}^n  P_i(p_i,q_i)$, as long as each $P_i$ has a hyperbolic equilibrium and a homoclinic orbit.

Of course, the potentials satisfying Morse non-degeneracy are generic.

 The phase space is $\cM:=\cD\times\T^n\times\cB\times\T^d$, endowed with the standard symplectic form.
The corresponding Hamilton's equations are
\begin{equation}\label{Haml_eq}
\begin{array}{ll}
	\dot{p}=-\frac{\partial H_0}{\partial q}-\vep\frac{\partial H_1}{\partial q},\\ \\
	\dot{q}=\frac{\partial H_0}{\partial p}+\vep\frac{\partial H_1}{\partial p},
\end{array}
\quad
\begin{array}{ll}
\dot{I}=-\vep\frac{\partial H_1}{\partial \vph},\\ \\
 \dot{\vph}=\frac{\partial H_0}{\partial I}+\vep\frac{\partial H_1}{\partial I}.
\end{array}
\end{equation}
 It is clear that the dynamics of the unperturbed system $H_0$ is integrable.  Hence, the diffusion phenomena may occur only  if $\vep\neq 0$.
Denoting the extended phase space
 $$\tcM:=\cD\times\T^n\times\cB\times\T^d\times\T,$$
the perturbation function $H_1$ in \eqref{main_H_system} is assumed to be real analytic on $\tcM$, which means the analyticity can extend to a complex neighborhood of $\tcM$.
 
 For each $\kappa>0$ we denote by $\tcM_\kappa$ the set of all points $(p,q,I,\vph,t)\in$ $\CC^n\times\CC^n/(2\pi \Z)^n$ $\times\CC^d\times\CC^d/(2\pi\Z)^d\times\CC/(2\pi\Z)$  satisfying 
\[\text{dist}(p, \cD)<\kappa, \quad \text{dist}(I, \cB)<\kappa,\quad |\text{Im}~\iota|<\kappa,~ \iota=q_i,\vph_i,t.
\]
$\tcM_\kappa$ is an open domain in the complex space. In order to discuss the genericity of Arnold diffusion in the real analytic category,
 we introduce the following space of bounded analytic functions on $\tcM_\kappa$,
\[
C^\omega(\tcM_\kappa):=\left\{f: \tcM_\kappa\to \CC ~\Big|~ f~ \text{is analytic}~,~ \sup_{x\in \tcM_\kappa} |f(x)|<\infty, ~ f(\tcM)\subset \R\right\}.
\]
Note that $f\in C^\omega(\tcM_\kappa)$ is real-valued on $\tcM$.
Clearly, $C^\omega(\tcM_\kappa)$ with the sup-norm $\|f\|_{\kappa}:=\sup_{z\in \tcM_\kappa} |f(z)|$ is a Banach space.
Sometimes, for simplicity, we  use $C^\omega_\kappa$ instead of $C^\omega(\tcM_\kappa)$ when there is no confusion.

\subsection{Main Result}

Now, we are ready to state our main result on the genericity.
In what follows,  Hamiltonian system \eqref{main_H_system} is always assumed to satisfy conditions {\bf(H1)--(H2)}. 

\begin{The}\label{Main_Thm1}
Given $\kappa>0$, $I_0\in\cB$ and a small neighborhood $V_{I_0}$ of $I_0$. Then there exists an open and dense set $\cU\subset C^\omega_\kappa$, and for each $H_1\in \cU$ we can find $\vep_0=\vep_0(H_1)>0$ and $\rho=\rho(H_1)>0$  satisfying the following property: for each $\vep\in(-\vep_0,\vep_0)\setminus\{0\}$, the Hamiltonian flow of $\HH_\vep=H_0+\vep H_1$  admits a trajectory  whose action variables $I(t)$  satisfy 
\[\sup_{t>0}\|I(t)-I(0)\|\geq \rho, \]
where the initial condition $I(0)\in V_{I_0}$.
\end{The}

\begin{Rem}
The above result states that we can find  one diffusing orbit whose 
initial condition $I(0)$ in the action space is just some point in the neighborhood $V_{I_0}$. This initial condition $I(0)$  in general may not be
  $I_0$.
To construct diffusion orbits one needs to pose some hypotheses on $H_1$, see \textbf{(H3a)-(H3b)} in Section \ref{section_mechanism}. 
Moreover,
	as we will show in Section \ref{section_proof}, the set of $H_1$ satisfying \textbf{(H3a)-(H3b)}  is indeed  $ C^\omega_\kappa$ dense and $C^3$ open.  These conditions of genericity in $H_1$ are rather explicit. 
They are given by conditions (which only require finite precision calculation) on a rapidly convergent integral. Hence, they can be verified in concrete models (e.g. in celestial mechanics). 
	Also, we point out that $\rho$ does not depend on $\varepsilon$, so Theorem \ref{Main_Thm1}  implies that for generic systems, 
	through an arbitrarily small neighborhood of a given point there passes a  trajectory whose action coordinates  go away from the initial values by  $O(1)$ with respect to the size of the perturbation. 
\end{Rem}
\begin{Rem}
$\kappa$ stands for the size of analytic extension.
It is worth noting that our result of analytic genericity holds for any $\kappa>0$.
\end{Rem}
\begin{Rem}\label{rem_smallV}
As we will see from the proof in the following sections,  our method also 
allows that $V_i$ are weakly hyperbolic, that is, $V_i$ can be replaced by $\delta V_i$ for a small $\delta>0$, and the perturbation parameter $\vep\ll \delta$. This is similar to Arnold's example \cite{Ar1964}.
Of course, in that case the threshold value $\vep_0(H_1)\ll \delta$
shrinks to zero as $\delta$ tends to zero.    This is, very typical of
the approches to diffusions near integrable systems. Of course one
expects that making $\vep$ larger will generate more diffusion. 
\end{Rem}

We can also interpret the  genericity result  stated above for system $H_0+H_1$ without using the parameter $\vep$. More precisely, let  $\mathfrak{S}$ be the unit sphere in the space $(C^\omega_\kappa, \|\cdot\|_\kappa)$ with $\kappa>0$. Then there exists a non-negative function $\epsilon_0: \mathfrak{S}\to [0,+\infty)$ taking positive values on an open-dense subset of $\mathfrak{S}$, such that for each $H_1$ in the $\epsilon_0$-ball
\[\mathfrak{B}=\Big\{\lambda P ~\Big|~ P\in \mathfrak{S},~\lambda\in \big(0,\epsilon_0(P)\big) \Big\},\]
the Hamiltonian $H_0+H_1$ admits Arnold diffusion.

We mention that our genericity result can  extend to the Hamiltonians of the form   $H_0(p,q,I)$ $+$ $\vep H_1(p,q,I,\vph,t;\vep)$ where $H_1(p,q,I,\vph,t;\vep)$ also depends analytically on the parameter $\vep$. 
In fact, our geometric method uses mainly the first-order analysis. The conditions \textbf{(H3a)-(H3b)}, see Section \ref{section_mechanism}, imposed on $H_1$  only involve the properties of $H_1(p,q,I,\vph,t;0)$.  Similar  discussions can also be found in works such as \cite{Gidea_Delallave_seara2020, Gidea_Delallave_seara2020_skip}. 

Let $C^\omega:=\bigcup_{\kappa>0}C^\omega_\kappa$. It is exactly the set of all  bounded real analytic functions on $\tcM$. Note that $C^\omega$ is a Fr\'echet space, we then have the following immediate consequence:
\begin{Cor}\label{Main_Cor2}
Given a point $I_0\in\cB$ and a small neighborhood $V_{I_0}$ of $I_0$. Then there exists an open and dense set $\cV\subset C^\omega$, and for each $H_1\in \cV$ we can find $\vep_0=\vep_0(H_1)>0$ and $\rho=\rho(H_1)>0$ satisfying the following property: for each $\vep\in(-\vep_0,\vep_0)\setminus\{0\}$, the Hamiltonian flow of $\HH_\vep=H_0+\vep H_1$ admits a trajectory whose action variables $I(t)$  satisfy
\[\sup_{t>0}\|I(t)-I(0)\|\geq \rho,\]
where the initial condition $I(0)\in V_{I_0}$.
\end{Cor}

As we will show in  Section \ref{section_proof}, the genericity 
is verified by taking advantage of perturbation functions depending only on $(q,\varphi, t)$. Thanks to the work of \cite{Gidea_Delallave_seara2020}, we
will see that two hypotheses formulated below as {\bf (H3a), (H3b)}
for a specific integral, imply diffusion. Hence, for us,
it suffices to show  that  {\bf (H3a), (H3b)} are generic. 
Therefore, we can even establish the genericity 
\emph{in the sense of Ma\~n\'e} \cite{Mane1996}, namely, the diffusive phenomenon occurs under generic periodic potential perturbations. More precisely, we denote by $C_\kappa^\omega(\T^{n+d+1})$ the set of all real analytic functions  which can extend analytically to the complex neighborhood 
$\{(q,\vph,t)\in\CC^{n+d+1}/(2\pi\Z)^{n+d+1} : |\textup{Im}~\iota|<\kappa, \iota=q_i,\vph_i,t\}.$ 
Then we have

\begin{The}\label{Main_Thm2}
Given $\kappa>0$, $I_0\in\cB$ and a small neighborhood $V_{I_0}$ of $I_0$. Then there exists an open and dense set $\cW\subset C^\omega_\kappa(\T^{n+d+1})$,  and for each  $P\in \cW$ we can find $\vep_0=\vep_0(P)>0$ and $\rho=\rho(P)>0$  satisfying the following property: for any $\vep\in(-\vep_0,\vep_0)\setminus\{0\}$, the Hamiltonian flow of $\HH_\vep=H_0+\vep P$  admits a trajectory  whose action variables $I(t)$  satisfy
\[\sup_{t>0}\|I(t)-I(0)\|\geq \rho,\]
where the initial condition $I(0)\in V_{I_0}$.
\end{The}

We end this section by giving a concluding  remark on our result and approach. 
\begin{enumerate} 
\item Our perturbative technique for the genericity is valid in both the $C^r$-differentiable $(3\leq r\leq \infty)$ and the $C^\omega$ topologies. Also, it  applies to the genericity in the sense of Ma\~n\'e.
	\item   The unperturbed part $H_0$ is only needed to be $C^r$ smooth with $r\geq 3$.  
	  We do not require the inner dynamics to satisfy a twist condition, and the diffusion mechanism used in the present paper only relies on the outer dynamics since invariant objects (e.g. primary KAM tori,  Aubry-Mather sets) of the inner map are not used at all.
	\item Both the phase space of the rotator and the phase space of the pendulums can be of arbitrary dimensions. 
	 \end{enumerate}

\subsection{Organization of the paper}
In Section \ref{section_mechanism}, we first review the
results we use on the normally hyperbolic invariant manifolds
and the scattering maps for the \textit{a priori} unstable
system \eqref{main_H_system}.
Then, we review the geometric
program established in \cite{Gidea_Delallave_seara2020}. It
allows us to obtain Arnold diffusion for the original
dynamics by shadowing the pseudo-orbits of the scattering
map. We provide more details for this geometric mechanism in
Appendix \ref{Appendix_section} for the reader's convenience.
Section \ref{section_proof} is devoted to the proofs of our
results on analytic genericity.  Appendix
\ref{appendix_Nor_hyp_theory} and Appendix
\ref{appendix_scattering_map} give general introductions to
the theory of NHIMs and the theory of scattering maps.
The perturbative argument to break the possible degeneracies of
the conditions in \cite{Gidea_Delallave_seara2020} is 
described in Section~\ref{section_proof}.

\section{Scattering maps and geometric mechanism of Arnold diffusion}\label{section_mechanism}
The main characteristic of an \textit{a priori} unstable Hamiltonian system is that there exists  a normally hyperbolic invariant manifold (NHIM)  with unstable and stable invariant manifolds.
The presence of these invariant objects plays an important role in the Arnold diffusion problem.  The scattering map of the NHIM is an effective tool to quantify the homoclinic excursions. This map associates the orbit asymptotic in the past to the orbit asymptotic in the future. 
Using the perturbation theory and the Melnikov method one can estimate  the effect of the perturbation on all the variables of the scattering map.
See Appendix \ref{appendix_Nor_hyp_theory} and Appendix 
\ref{appendix_scattering_map} for general introductions.

In this section, we first give some important results on the NHIM and the scattering map for our \textit{a priori} unstable system $\HH_\vep=H_0+\vep H_1$. Then, we review a recent geometric mechanism of Arnold diffusion established in \cite{Gidea_Delallave_seara2020}. 

Recall that $\HH_\vep$ satisfies conditions \textbf{(H1)--(H2)}.  
From now on, it is convenient to fix two closed balls (suitably large) $\cD_*\subset \cD$ and $\cB_*\subset \cB$, and  study the dynamics  on the following domain
\[(p, q, I, \vph)\in \cD_*\times\T^n\times\cB_*\times\T^d.\]

\subsection{Normal hyperbolicity of the unperturbed system}
As the unperturbed system $H_0$ is given by
\begin{equation*}
H_0(p,q,I)=h(I)+\sum_{i=1}^n \pm\Big[\frac{1}{2}p_i^2+V_i(q_i)\Big],
\end{equation*}
we use $\Phi_{t,0}$ to denote the corresponding autonomous $C^{r-1}$ Hamiltonian flow on $\cM=\cD\times\T^n\times\cB\times\T^d$. Here, the subscript ``0" represents $\vep=0$. 

Condition \textbf{(H2)} implies that each pendulum $\frac{1}{2}p_i^2+V_i(q_i)$   has two homoclinic orbits. For each $i$,  we choose and fix one homoclinic orbit  $(p_i^0(t)$, $q_i^0(t))$. It converges exponentially to the hyperbolic equilibrium $(0,0)$ with characteristic exponent 
\begin{equation}\label{char_exp}
	\lambda_i:=\sqrt{-V_i''(0)}>0.
\end{equation}
This is equivalent to saying 
\[\text{dist}\Big(\big(p_i^0(t), q_i^0(t)\big), \big(0,0\big)\Big)\leq C e^{-\lambda_i|t|},\quad\text{as}~ t\longrightarrow\pm \infty.\]

The autonomous flow $\Phi_{t,0}$  has a $2d$-dimensional invariant manifold $\La_0$ with boundary,
\begin{align*}
	\La_0=\{(0,0,I,\vph)~:~ (I,\vph)\in\cB_*\times\T^d\}.
\end{align*}
$\La_0$ is foliated completely by invariant tori, and hence the dynamics restricted on $\La_0$ is integrable. 
Sometimes, we need to work in the extended space $\tcM=\cM\times\T$, which yields a $(2d+1)$-dimensional manifold  
\[\tLa_0=\{(0,0,I,\vph, t)~:~ (I,\vph,t)\in\cB_*\times\T^d\times\T\} \subset \tcM.\]
$\tLa_0$ is invariant under the extended flow $\tPh_{t,0}$.  

$\tLa_0$ is a normally hyperbolic invariant manifold (see Appendix \ref{appendix_Nor_hyp_theory}). To verify it, we use \eqref{char_exp} and take the normal exponents  
\begin{align}\label{exponents_H_0}
	\lambda_s=-\max_{i=1,\cdots,n}\lambda_i, \quad \mu_s=-\min_{i=1,\cdots,n}\lambda_i;\quad \lambda_u=\min_{i=1,\cdots,n}\lambda_i, \quad \mu_u=\max_{i=1,\cdots,n}\lambda_i,
\end{align}
For the central exponents, we can take $-\lambda_c=\mu_c$ with the positive exponent $\mu_c$ as close as desired to the value $0$ since the dynamics on $\tLa_0$ is completely integrable.   Consequently, for every $\tilde{x}\in\tLa_0$ we have the invariant splitting of the tangent bundle
$T_{\tilde{x}}\tcM=T_{\tilde{x}}\tLa_0\oplus E_{\tilde{x}}^s\oplus E_{\tilde{x}}^u$, and
\begin{align}\label{expansion_contraction_rate}
    v\in E_{\tilde{x}}^s  & \Longleftrightarrow C^{-1}e^{t\lambda_s }\|v\|\leq\|D\tPh_{t,0}(\tilde{x})v\|\leq Ce^{t\mu_s }\|v\|,  \quad t\geq 0,\nonumber \\
    v\in E_{\tilde{x}}^u & \Longleftrightarrow  C^{-1}e^{t\mu_u }\|v\|\leq  \|D\tPh_{t,0}(\tilde{x})v\|\leq Ce^{t\lambda_u }\|v\|,  \quad t\leq 0,\\ 
    v\in T_{\tilde{x}}\tLa_0 & \Longleftrightarrow C^{-1}e^{|t|\lambda_c }\|v\| \leq \|D\tPh_{t,0}(\tilde{x})v\|\leq Ce^{|t|\mu_c }\|v\| ,  \quad t\in\R, \nonumber
\end{align}
where  the constant $C>1$. The stable (resp. unstable) space $E_{\tilde{x}}^{s}$  (resp. $E_{\tilde{x}}^{u}$) is just the direct sum of the stable (resp. unstable) spaces at the hyperbolic equilibrium of each pendulum. 
In particular,  $\La_0$ is also a NHIM of the flow $\Phi_{t,0}$, with the same exponents $\lambda_s\leq\mu_s<$ $\lambda_c<0<\mu_c$ $<\lambda_u\leq \mu_u$.

On the other hand,  there is also a family of homoclinic orbits parameterized by
\begin{equation}\label{homo_orbits}
	\begin{split}
		p^0(\tau+t\bar{1})=\big(p_1^0(\tau_1+t),\cdots,p_n^0(\tau_n+t)\big),\qquad
q^0(\tau+t\bar{1})=\big(q_1^0(\tau_1+t),\cdots,q_n^0(\tau_n+t)\big),
	\end{split}
\end{equation}
where $\tau=(\tau_1,\cdots,\tau_n)\in\R^n$ and $\bar{1}=(1,\dots,1)\in\R^n$. Each parameter $\tau_i$, $i\in\{1,\cdots, n\}$, represents the time shift for the $i$-th homoclinic orbit.  $\big(p^0(\tau+t\bar{1}),q^0(\tau+t\bar{1})\big)$ is asymptotic to  $(0, 0)$ in the future with an exponential rate at least $\mu_s$, and in the past with an exponential rate  at least $\lambda_\mu$. Moreover, these homoclinic orbits form the stable  manifold $W^s_{\tLa_0}$ and the unstable manifold $W^u_{\tLa_0}$ of the NHIM $\tLa_0$. In particular,  the unstable and stable manifolds coincide, that is $W^s_{\tLa_0}=W^u_{\tLa_0}$.

\subsection{Persistence of normally hyperbolic invariant manifolds} 
In the theory of normally hyperbolic invariant manifolds,  it is well known that the NHIM
  along with its stable and unstable manifolds persist under small perturbations  \cite{Fen1971,Fen1979,HPS1977}. In general, the NHIM will only be finitely differentiable. The optimal regularity depends on the ratio of the normal exponents and the central exponents.
  
 For the perturbed system $\HH_\vep=H_0+\vep H_1$ with $\vep\neq 0$, we have the non-autonomous Hamilton's equations \eqref{Haml_eq}.  By supplementing equations \eqref{Haml_eq} with the equation $\dot{s}=1$, 
we can  consider the extended flow, denoted as $\tPh_{t,\vep}$, associated with the Hamiltonian $\HH_\vep(p,q,I, \vph, s)$. Then, $\tPh_{t,\vep}$ becomes a $C^{r-1}$ autonomous flow on $\tcM$.
  
  Following Appendix \ref{appendix_subsub_smooth}, we set
  \begin{equation}\label{regularity_index}
	 \ell=\min\{\ell_u,\ell_s\}
\end{equation}
where
\[\ell_u=\max\left\{k=1,\cdots,r-1~:~k<\frac{\mu_s}{\lambda_c}\right\}   \quad\text{and}\quad
	\ell_s=\max\left\{k=1,\cdots,r-1~:~k<\frac{\lambda_u}{\mu_c}\right\}
\]
 and the exponents $\lambda_s\leq\mu_s<$ $\lambda_c<0<\mu_c$ $<\lambda_u\leq \mu_u$ are given in \eqref{expansion_contraction_rate}.

\begin{Rem} 
Note that  $\ell_s$ and $\ell_u$  are  only finite even when $r=\infty$
 or $\omega$. Taking the central  exponents $\lambda_c$ and $\mu_c$ sufficiently small if necessary, we can always let the indices $\ell_s\geq 2, \ell_u\geq 2$, and hence \[\ell\geq 2.\]
 In particular, in the case of $r\in[3,\infty)$, we can have $\ell=\ell_s=\ell_u=r-1$ for $\vep$ sufficiently small since $\lambda_c$ and $\mu_c$ can be chosen as close as desired to $0$.

 The argument of \cite{Gidea_Delallave_seara2020}, is a transversality
 argument that only requires a few derivatives of the invariant manifolds
 and the perturbations involved. 
 
 \end{Rem}
  
\begin{Pro}\label{Pro_NHIM_H_ep}
Let $\HH_\vep$ satisfy  conditions {\bf(H1)--(H2)}. Then there exists $\vep_0>0$ sufficiently small such that for each $\vep\in(-\vep_0,\vep_0)$, the flow $\tPh_{t,\vep}$ has a normally hyperbolic locally invariant manifold $\tLa_\vep$ with the associated stable manifold $W^{s}_{\tLa_\vep}$ and unstable manifold $W^{u}_{\tLa_\vep}$. 
Moreover, the manifolds $\tLa_\vep$ and $W^{u,s}_{\Lambda_\vep}$ are  $C^{\ell}$ differentiable with the index $\ell$ given in \eqref{regularity_index}, and $\tLa_\vep$ is diffeomorphic to $\tLa_0$.
\end{Pro}

We give a sketch of the proof of Proposition \ref{Pro_NHIM_H_ep} for the reader's convenience. We also refer to  \cite{DLS2006, DLS2016, Gidea_delaLlave2018} for more details.

To prove Proposition \ref{Pro_NHIM_H_ep}, we first point out that the NHIM $\tLa_0$ of the unperturbed equations has non-empty boundary on which the flow $\tPh_{t,0}$ is invariant, but the invariance on the boundary will be destroyed under perturbations in general. Then, just as pointed out in \cite{Fen1971,Fen1977}, a standard treatment is to  consider a slightly modified Hamiltonian.  
More precisely, we take two open domains $U_1$ and $U_2$ close enough to $\cD$ and $\cB$ respectively, and
 $\cD_*\subset U_1\subset\cD$ and $\cB_*\subset U_2\subset\cB$. Let
 $\rho(p, I): \R^n\times\R^d\longrightarrow [0,1]$ be a $C^\infty$ smooth bump function such that $\rho|_{U_1\times U_2}\equiv 1$, and $\rho(p, I)=0$ for those points $(p, I)$ outside $\cD\times\cB$. 
 Then we define the modified Hamiltonian $G_\vep$ as follows
 \begin{align*}
 	G_\vep:= H_0+\vep \rho H_1.
 \end{align*}
 Clearly,  $G_\vep=H_0$ for $(p, I)\notin\cD\times\cB $,  
 and $G_\vep=\HH_\vep$ on $U_1\times\T^n\times U_2\times\T^d\times\T$. 
We use $\tPh_{t, G_\vep}$ to denote the associated Hamiltonian flow.
 Then, for $\vep=0$  the   manifold
$\tLa_{G_0}=\{(0,0,I,\vph, s)~:~ (I,\vph,s)\in\R^d\times\T^d\times\T\}$
is normally hyperbolic and invariant under the flow $\tPh_{t,G_0}$. Note that $\tLa_{G_0}$ has no boundary. Thus we can apply Theorem \ref{appendix_persistence} to the perturbed system $G_\vep$ to obtain a unique NHIM $\tLa_{G_\vep}$ of the flow $\tPh_{t, G_\vep}$.

In general, the manifold $\tLa_{G_\vep}$ constructed in this way  depends on the nature of the modification on the boundary. Anyway, orbits that never pass through the modified region behave identically to those of the unmodified Hamiltonian $\HH_\vep$.  
Since $\tPh_{t, G_\vep}$ agrees with the original flow $\tPh_{t, \vep}$ on the domain $U_1\times\T^n\times U_2\times\T^d\times\T$, we therefore obtain a normally hyperbolic manifold $\tLa_{\vep}$ that is locally invariant under $\tPh_{t, \vep}$. Note that $\tLa_{\vep}$  is in general not unique, as its construction depends on the modified Hamiltonian $G_\vep$. Nevertheless, any one of them can be used to prove our following results because the conditions {\bf(H1)--(H2)} are only on the unperturbed part $H_0$.   The choice only affects the smallness of $\vep_0$. See also \cite{DLS2008} for more discussion. 
 
The next step is to check  the smoothness of $\tLa_\vep$. Observe that
 $\tLa_\vep$ is normally hyperbolic with slight changes on the normal and central exponents given in \eqref{expansion_contraction_rate}. We denote the perturbed exponents by
 \begin{equation}\label{per_exponents}
 	\lambda_{s,\vep}\leq\mu_{s,\vep}<\lambda_{c,\vep}<0<\mu_{c,\vep}<\lambda_{u,\vep}\leq \mu_{u,\vep}.
 \end{equation}
 They are $O(\vep)$-close to those in \eqref{expansion_contraction_rate}. This implies that the index $\ell$ defined in \eqref{regularity_index} would remain unchanged as long as $\vep$ is small enough.  Hence, $\tLa_\vep$ is  $C^{\ell}$ smooth.

Finally, with the persistent manifold $\tLa_\vep$,  we obtain the local stable manifold $W^{s,loc}_{\tLa_\vep}$ and the local unstable manifold $W^{u,loc}_{\tLa_\vep}$, which can be prolonged to  $W^{s}_{\tLa_\vep}$ and  $W^{u}_{\tLa_\vep}$, respectively. This finishes the proof of Proposition \ref{Pro_NHIM_H_ep}.

\subsection{Transversal intersections and scattering maps}
The theory of scattering maps is developed to quantify the homoclinic excursions. To define the scattering map,  some transversal intersection hypotheses are needed, see \eqref{tran_intersection_condition_1}-\eqref{tran_intersection_condition_2} in Appendix \ref{appendix_scattering_map}. We can use the Melnikov method to measure these transversal intersection property  in the perturbed system.

In our model, for $\vep=0$ the stable and unstable manifolds of the flow  $\tPh_{t,0}$  coincide:
\begin{align*}
W^s_{\tLa_0}=W^u_{\tLa_0}=\{(p^0(\tau),q^0(\tau), I,\vph, s)~:~\tau\in \R^n, (I,\vph,s)\in  \cB_*\times \T^{d+1} \}	
\end{align*}
In the case of $\vep\neq 0$,  $W^u_{\tLa_\vep}$ and  $W^s_{\tLa_\vep}$  do not coincide in general, and possibly do not intersect transversely along  homoclinic manifolds.  To measure the splitting of manifolds we    introduce the  \emph{Poincar\'e function} (or \emph{Melnikov potential}),
\begin{equation}\label{Def_PM_function}
\begin{split}
L(\tau, I,\vph,s):=-\int^\infty_{-\infty}\Big[& H_1\left(p^0(\tau+t\bar{1}),q^0(\tau+t\bar{1}),I,\vph+\omega(I)t, s+t\right)-H_1\left(0,0,I,\vph+\omega(I)t, s+t\right)\Big]\,dt,
\end{split}
\end{equation}
where $\omega(I)=(\omega_1(I),\cdots,\omega_n(I))=D h(I)\in C^{r-1}$,  $\bar{1}=(1,\dots,1)\in\R^n$ and the orbits $(p^0, q^0)$ is given in \eqref{homo_orbits}. 
It is a convergent improper integral of the perturbation evaluated along homoclinic orbits of the unperturbed system. 
We  stress that this integral  is absolutely convergent, because 
$\big(p^0(\tau+t\bar{1}),q^0(\tau+t\bar{1})\big)$ converges exponentially fast to $(0,0)$ as $t\to\pm\infty$. 

By definition it is easily seen that
\begin{equation}\label{periodicity_L}
L(\tau+\sigma\bar{1},I,\vph,s)=L(\tau,I,\vph-\omega(I)\sigma,s-\sigma), \quad \text{for all~}\sigma\in\R
\end{equation}
In particular, for the lower-dimensional case $n=1$, $L(\tau,I,\vph,s)$ $=$ $L(0,I,\vph-\omega(I)\tau,s-\tau)$ for all $\tau\in\R$, which  implies that the  Poincar\'e function $L(\tau,I,\vph,s)$ is periodic or quasi-periodic with respect to $\tau$.
\begin{Rem}[\textbf{Regularity of $L$}]
 Note that  the integral \eqref{Def_PM_function} is evaluated not on the perturbed homoclinic orbit but only on the unperturbed one.
As $H_1$ is real analytic,
it is not difficult to check that
 $L(\tau,I,\vph,s)$ is  $C^{r-1}$ smooth. More precisely,  the dependence on the variable $\tau$ is $C^r$,  the dependence on the variable $I$ is $C^{r-1}$ and  the dependence on the variables $(\vph, s)$ is $C^\omega$. In particular, in the case when $r=\infty$ (resp. $\omega$), the function $L(\tau,I,\vph,s)$ is also $C^\infty$  (resp. $C^\omega$).
 \end{Rem}

 To verify the mechanism in \cite{Gidea_Delallave_seara2020}, it 
 suffices to verify two assumptions on $L$. Assumption {\bf(H3a)} predicts that
 the perturbations generate homoclinic intersections and
 assumption {\bf(H3b)}  implies that the homoclinic  intersections indeed
 generate changes in the action.  Rather remarkably, both assumptions
 amount to properties of $L$. 

 \smallskip 

 The non-degenerate critical points of $L$ would yield the existence of transverse homoclinic orbits for the perturbed system.

Given a point $(I,\vph,s)\in \cB_*\times\T^{d}\times\T$ and assume the map $\tau\in\R^n$ $\longmapsto$ $L(\tau,I,\vph,s)$ has a non-degenerate critical point at $\tau^*$, that is 
\[ \frac{\partial L}{\partial \tau}(\tau^*, I,\vph,s)=0,\quad \frac{\partial^2 L}{\partial \tau^2}(\tau^*, I,\vph,s)\neq 0.\]
Then for $0<|\vep|\leq \vep_0$ small enough, there exists a locally unique $z^*$ of the form
\begin{equation*}
	z^*=z^*(\tau^*,I,\vph,s)=(p^0(\tau^*) +O(\vep), q^0(\tau^*)+O(\vep), I,\vph,s )
\end{equation*}
such that the unstable manifold $W^u_{\tLa_\vep}$ and the stable manifold $W^s_{\tLa_\vep}$ intersect transversally at $z^*$, i.e. 
\[
	T_{z^*}\tcM=T_{z^*} W^u_{\tLa_\vep}+T_{z^*}W^s_{\tLa_\vep}.
\]

This therefore leads us to formulate  the following assumption \textbf{(H3a)}.

\begin{itemize}
	\item [\bf(H3a)] there exists an open neighborhood $U^-$ of a point $(I_0,\vph_0,s_0)$, where
	 $U^-:=\mathcal{I}\times\mathcal{J}$ with $\mathcal{I}$ a ball in $\cB_*$ and $\mathcal{J}$ an open set in $\T^{d}\times\T$, such that for each point $(I,\vph,s)\in U^-$, the map 
$$\tau\in\R^n\longmapsto L(\tau,I,\vph,s)$$
has a non-degenerate critical point $\tau^*$. 
By the implicit function theorem, $\tau^*$ is locally given by a $C^{r-1}$ function 
$$\tau^*=\tau^*(I,\vph,s).$$ 
\end{itemize}

We stress that in the present paper the size of the domain $U^-$ in
{\bf(H3a)} does not need to be too large, but it is independent of
$\vep$ since the expression we need
to study does not involve $\vep$.
In applications, it suffices to verify the existence of
non-degenerate critical point for a fixed point $(I_0,\vph_0,s_0)$,then by  the implicit function theorem there is a neighborhood $U^-$ of
$(I_0,\vph_0,s_0)$ whose size is independent of $\vep$, such that
{\bf(H3a)} holds. 

In the case when $n=1$, the identity $L(\tau,I,\vph,s)=L(0,I,\vph-\omega(I)\tau, s-\tau)$ holds. This implies that the one-dimensional map $\tau\longmapsto L(\tau,I_0,\vph_0,s_0)$ always has critical points for each fixed $(I_0,\vph_0,s_0)$. 
In the case when  $n>1$,  for each $I_0$ there exists $(\vph_0,s_0)$ such that
the map $\tau\in\R^n\longmapsto L(\tau,I_0,\vph_0,s_0)$  has critical points \cite{Delshams_Huguet2011, Gidea_delaLlave2018}.
Hence, the only content of {\bf (H3a)} is
that some of these critical  points are non-degenerate. 

With Proposition \ref{Pro_NHIM_H_ep}, the following result  is well known. See for instance \cite{DLS2006, DLS2016, Gidea_delaLlave2018} for more details.
\begin{Pro}\label{pertur_Wu_Ws_Gamma}
Let the Hamiltonian \eqref{main_H_system} satisfy conditions {\bf(H1)}, {\bf(H2)} and {\bf(H3a)}. Then for each $\vep\in(0, \vep_0)$ with $\vep_0$ small enough, there exists a $C^{\ell}$ homoclinic manifold $\tGa_\vep$ $\subset$ $W^s_{\tLa_\vep}\cap W^u_{\tLa_\vep}$, which can be parameterized by
\[
\tGa_\vep=\left\{  z^*\big(\tau^*(I,\vph,s),I,\vph,s\big)\,\,:\,\, (I,\vph,s)\in U^-\right\},
\]
and the stable and unstable manifolds $W^{s,u}_{\tLa_\vep}$  intersect transversally along $\tGa_\vep$. 
Also, $\tGa_\vep$ is transverse to the foliations of the stable/unstable
manifolds.
Moreover, $\tGa_\vep$  can $C^{\ell}$-smoothly extend to a manifold $\tGa_0:=$ $\left\{ (p^0(\tau^*), q^0(\tau^*),I,\vph,s)\,:\, (I,\vph,s)\in U^-\right\}$ as $\vep\to 0$. 
\end{Pro}

As we can see from   Remark  \ref{existence_homo_mfld} in Appendix \ref{appendix_scattering_map},
taking $U^-$ suitably small if necessary  we can ensure the transversality conditions \eqref{tran_intersection_condition_1}--\eqref{tran_intersection_condition_2} are satisfied along $\tGa_\vep$. 
Meanwhile, we consider the $C^{\ell}$ wave maps
\[\Omega^\vep_+: W^s_{\tLa_\vep}\longrightarrow \tLa_\vep\quad\text{and}\quad \Omega^\vep_-: W^u_{\tLa_\vep}\longrightarrow \tLa_\vep,\] which are projections along the stable and unstable leaves (see Appendix \ref{appendix_scattering_map}).  Then,  the wave maps restricted on $\tGa_\vep$, namely $\Omega^\vep_\pm\big|_{\tGa_\vep}$, are diffeomorphisms. In addition, by recalling the normal exponents in \eqref{per_exponents},  for each $\tilde{x}\in\tGa_\vep$ we can find a unique point $\tilde{x}_+=\Omega^\vep_+(\tilde{x})$ and a unique point $\tilde{x}_-=\Omega^\vep_-(\tilde{x})$ satisfying 
\begin{align*}
	\text{dist}\big(\tPh_{t,\vep}(\tilde{x}), \tPh_{t,\vep}(\tilde{x}_+)\big)\leq Ce^{t\mu_{s,\vep}} \quad\text{as~} t\longrightarrow +\infty, \qquad
	\text{dist}\big(\tPh_{t,\vep}(\tilde{x}), \tPh_{t,\vep}(\tilde{x}_-)\big)\leq Ce^{t\lambda_{u,\vep}}\quad\text{as~}t\longrightarrow -\infty.
\end{align*}    

Consequently, each $\tGa_\vep$ with $0<|\vep|\leq\vep_0$ is a \emph{homoclinic channel}. This enables us to define the  scattering map $\tsig_\vep:= \sigma^{\tGa_\vep}$ associated to the homoclinic channel $\tGa_\vep$, that is
\begin{equation*}
\begin{aligned}
	\tsig_\vep= \Omega^\vep_+\Big|_{\tGa_\vep}\circ\left(\Omega^\vep_-\Big|_{\tGa_\vep}\right)^{-1}~&:~\Omega^\vep_-(\tGa_\vep)\longrightarrow \Omega^\vep_+(\tGa_\vep)\\
	\tilde{x}_- &\longmapsto \tilde{x}_+.
\end{aligned}	
\end{equation*}

The scattering map $
\tsig_\vep$ above is a $C^{\ell}$ diffeomorphism.  By Proposition \ref{pertur_Wu_Ws_Gamma}  the homoclinic manifold $\tGa_\vep$ can extend smoothly to a limiting manifold $\tGa_0$.
Even $\tGa_0$ is not a transversal intersection,  $\tsig_\vep$ can still $C^{\ell}$-smoothly  extend to  the identity map $\tsig_0=\textup{Id}$, as $\vep\to 0$.

\subsection{Perturbative formulas for the scattering maps}
The Melnikov method can also be used to estimate the effect of the perturbations on the scattering map.
As was shown in \cite{DLS2008},  the symplectic property allows to give  perturbative formulas for the Hamiltonian which generates the deformation of a family of symplectic scattering maps.  This requires the dimension of the NHIM to be even while our $\tLa_\vep$ mentioned above is of odd dimensions. To overcome this difficulty, it is standard to consider an autonomous Hamiltonian defined by
\begin{align}\label{even_dimen_Ham}
	\wtH_\vep(p,q,I,\vph, A, s)=A+\HH_\vep(p,q,I,\vph, s)=A+H_0(p,q, I)+\vep H_1(p,q,I,\vph, s),
\end{align}
where $(A, s)\in\R\times\T$ are symplectically conjugate variables. 
The extended phase space is endowed with the standard symplectic structure $\widetilde{\omega}=\sum_{i=1}^n dp_i\wedge dq_i+ \sum_{i=1}^d dI_i\wedge d\vph_i+dA\wedge d s$. Then, the motions of the conjugate variables $(A, t)$ are governed by 
\[\dot{A}=-\partial_s \HH_\vep(p,q,I,\vph, s),\quad \quad \dot{s}=1.\]

However, the variable $A$ does not play any dynamical role, because $A$ does not appear in any of the ODEs for any of the coordinates, including itself.  
Consequently, by abuse of notation, we continue to use $\tPh_{t,0}$ to denote the unperturbed  flow and use
\[\tLa_0=\{(0,0,I,\vph, A, s)~:~ (I,\vph, s)\in\cB_*\times\T^d\times\T, A\in\R\}, \]
to denote the normally hyperbolic invariant manifold, which is $(2d+2)$-dimensional. Since $A$ does not play any dynamical role, the results obtained in the previous sections remain true for $\wtH_\vep$. Then we continue to use $\tLa_\vep$ and $W^{s,u}_{\tLa_\vep}$, respectively, to denote the NHIM and the associated stable and unstable manifolds for the flow $\tPh_{t,\vep}$. Also, we have the scattering map $\tsig_\vep:= \sigma^{\tGa_\vep}$ associated to the homoclinic channel $\tGa_\vep$. Now that $\tLa_\vep$ has even dimensions, the map $\tsig_\vep$ is symplectic (see \cite{DLS2008}). 

The perturbed NHIM $\tLa_\vep$ can be described in terms of the coordinates $(I, \vph, A, s)\in\tLa_0$. In fact,
there is a unique $C^{\ell}$-smooth family of symplectic parametrization $\tk_\vep$ $:\tLa_0\to \tLa_\vep$, with $\tk_0=\textup{Id}$, 
 satisfying 
\begin{align*}
	 \tk^*_\vep\widetilde{\omega}=\tk_0^*\widetilde{\omega},\quad
 \quad	\frac{d}{d\vep}\tk_\vep\in E^{s,\vep} \oplus E^{u,\vep}.
\end{align*}
 Then
 we can express the scattering map $\tsig_\vep$ on the reference manifold $\tLa_0$ by:
\begin{equation*}
 	\wts_\vep=\tk_\vep^{-1}\circ\tsig_\vep\circ \tk_\vep~: ~\tk^{-1}_\vep\big(\Omega^\vep_-(\tGa_\vep)\big)\subset \tLa_0 \longrightarrow \tk^{-1}_\vep\big(\Omega^\vep_+(\tGa_\vep)\big) \subset\tLa_0.
 \end{equation*}
It is clear that  $\wts_\vep$ are the expression of $\tsig_\vep$
in the same coordinate system $\tLa_0$, and hence $\wts_\vep\in C^\ell$ with $\ell\geq 2$.
Moreover, using the  deformation theory 
 this family of symplectic maps $\wts_\vep$ can be generated by a Hamiltonian vector field \cite{DLS2008}: there exists a Hamiltonian function $\cS_\vep$ such that
\begin{align*}
	\frac{d}{d\vep}\wts_\vep=\mathcal{X}_\vep\circ \wts_\vep,\qquad \iota_{\mathcal{X}_\vep} \widetilde{\omega}_0=d\cS_\vep,
\end{align*}
where $\widetilde{\omega}_0:=\tk_0^*\widetilde{\omega}=\sum_{i=1}^d dI_i\wedge d\vph_i+dA\wedge ds$. The Hamiltonian
$\cS_\vep=\cS_0+O(\vep)$ and 
\[\cS_0(I,\vph,A, s):=-L\left(\tau^*(I,\vph,s),I,\vph,s\right).\]
Here, $L$ is the Melnikov potential and $\tau^*$ is given in $\textbf{(H3a)}$. The function $L(\tau^*(I,\vph,s), I, \vph, s):$ $U^-\to\R$ is defined in a domain  $U^-$ whose size is independent of $\vep$. Thus, we compute the perturbed scattering map $\wts_\vep$ up to the first order with respect to the size of the perturbation:
\begin{equation}\label{deformation}
\wts_\vep(I,\vph, A, s)=(I,\vph, A, s)+\vep \mathbf{J} \nabla\cS_0+O(\vep^{2}), \quad\quad (I,\vph, s)\in U^-, ~ A\in\R
\end{equation}	
with $\mathbf{J}$ the canonical matrix of the symplectic form $\widetilde{\omega}_0$. See \cite{DLS2008}.

We infer from $\frac{\partial L}{\partial \tau}(\tau^*(I, \vph, s),$ $I,\vph,s)=0$ that
\begin{align*}
	\frac{\partial\cS_0}{\partial I}(I,\vph,A, s)=-\frac{\partial L}{\partial\tau}(\tau^*,I,\vph,s)\, \frac{\partial \tau^*}{\partial I}- \frac{\partial L}{\partial I}(\tau^*,I,\vph,s)=-\frac{\partial L}{\partial I}(\tau^*,I,\vph,s).
	\end{align*}
Similarly, 
\begin{align*}
	\frac{\partial\cS_0}{\partial \vph}=-\frac{\partial L}{\partial\vph}(\tau^*,I,\vph,s),\qquad
	\frac{\partial\cS_0}{\partial A}=0, \qquad \frac{\partial\cS_0}{\partial s}=-\frac{\partial L}{\partial s}(\tau^*,I,\vph,s).
	\end{align*}
Here and subsequently, we use $\frac{\partial L}{\partial I}$ to denote the partial derivative of the Poincar\'e function $L(\tau, I, \vph, s)$ with respect to the second variable, $\frac{\partial L}{\partial \vph}$ to denote the partial derivative with respect to the third variable, and $\frac{\partial L}{\partial s}$ to denote the partial derivative with respect to the fourth variable. Now, formula \eqref{deformation} becomes 
\begin{align}\label{SMap}
	\wts_\vep(I, \vph, A, s)=\Big(I+\vep\frac{\partial L}{\partial\vph}(\tau^*, I,\vph,s),\vph-\vep\frac{\partial L}{\partial I}(\tau^*, I,\vph,s), A+\frac{\partial L}{\partial s}(\tau^*, I,\vph,s), s\Big)+O(\vep^{2}),
\end{align}
for $(I,\vph, s)\in U^-$ and $A\in\R$.

Next, we introduce the auxiliary function $\cL(I,\vph,s):=L(\tau^*(I, \vph, s),\vph, s )$. We claim that 
\begin{align}\label{L_shift}
	\cL(I,\vph-\omega(I)s,0)=\cL(I,\vph,s),\qquad (I,\vph, s)\in U^-.
\end{align}
Indeed, if $\tau^*(I,\vph,s)$ is a critical point of the map $\tau\in\R^n\longmapsto L(\tau,I,\vph,s)$, then $\tau^*(I,\vph,s)-s\bar{1}$ with $\bar{1}=(1,\dots,1)\in\R^n$ is a critical point of the map 
$\tau\in\R^n\longmapsto L(\tau+s\bar{1},I,\vph,s).
$
Thanks to \eqref{periodicity_L}, $L(\tau+s\bar{1},I,\vph,s) =L(\tau,I,\vph-\omega(I)s,0)$,
which implies
\begin{equation*}
	\tau^*(I,\vph-\omega(I)s,0)=\tau^*(I,\vph,s)-s\bar{1}.
\end{equation*}
Then our claim \eqref{L_shift} follows immediately.
This enables us to introduce the new variables
\[\theta=\vph-\omega(I) s,
\]
and define the \emph{reduced Poincar\'e function},  
\begin{equation*}
\cL^*(I,\theta):=\cL(I,\vph-\omega(I) s,0)=\cL(I,\vph, s).
\end{equation*}
The reduced Poincar\'e function is well defined on the following $2d$-dimensional set
\[
\text{Dom}(\cL^*)=\{(I,\theta)\in\R^d\times\T^d ~:~ \textup{there exists~} s\in\T \text{~such that~} (I, \theta+\omega(I)s, s)\in U^- \}.
\]

\begin{Rem}\label{rem_par_der}
By what we have shown above,  
\[
\cL^*(I,\theta)=L(\tau^*(I,\theta,0),I,\theta,0)~:~\text{Dom}(\cL^*)\longrightarrow \R.
\]
Since $\tau^*(I,\theta,0)$ is a critical point for the map $\tau$ $\mapsto$ $L(\tau, I, \theta, 0)$, we deduce 
\begin{align*}
	\frac{\partial \cL^*}{\partial I}(I, \theta)= \frac{\partial L}{\partial I}(\tau^*(I,\theta,0),I,\theta,0)\quad\text{and}\quad \frac{\partial \cL^*}{\partial \theta}(I, \theta)= \frac{\partial L}{\partial \theta}(\tau^*(I,\theta,0),I,\theta,0).
\end{align*}	
\end{Rem}

The reduced Poincar\'e function $\cL^*(I,\theta)$ plays a crucial role in understanding the scattering map of the time-$2\pi$ map of the Hamiltonian flow, see Proposition \ref{Pro_SM_f_vep} below.  To explain this, we recall the extended Hamiltonian $\wtH_\vep=A+\HH_\vep$ in \eqref{even_dimen_Ham} which is $2\pi$-periodic in the variable $s$. Fix an energy $\wtH_\vep=E$ for some $E$, and then restrict to a Poincar\'e section $\{s=2\pi\}$ for the Hamiltonian flow $\tPh_{t,\vep}$. Hence we obtain a $2(n+d)$-dimensional space in $\cM$ with the coordinates $(p, q, I, \vph)$, and use  $f_\vep$ to denote the first return map.  $f_\vep$ is exactly the time-$2\pi$ map for the flow of the Hamiltonian $\HH_\vep$, so $f_\vep$ is symplectic.
The $2d$-dimensional manifold $
	\La_0=\{(0,0,I,\vph)~:~ (I,\vph)\in\cB_*\times\T^d\}
$ 
is a NHIM for $f_0$. Under perturbations, there are manifolds $\La_\vep$ and symplectic parameterizations $k_\vep$ such that \[k_\vep : \La_0\to \La_\vep,\quad k_\vep(\La_0)=\La_\vep.\]
In fact, $\La_\vep$ is exactly, by omitting the variable $A$ since it does not play any dynamical role, the intersection of $\tLa_\vep$ and the section $\{s=2\pi\}$. $k_\vep$ is exactly the  restriction of $\tk_\vep$ to the section $\{s=2\pi\}$. Thus, $\Lambda_\vep$ is the normally hyperbolic (locally) invariant manifold for the map $f_\vep$. 

Analogously, $f_\vep$ has the corresponding homoclinic channel $\Gamma_\vep$ and the scattering map $\sigma_\vep:= \sigma^{\Gamma_\vep}$
$:$ $\Omega^\vep_-(\Gamma_\vep)$ $\longrightarrow \Omega^\vep_+(\Gamma_\vep)$. As mentioned before, it is more convenient to describe the scattering map in the same coordinate $(I, \theta)$ via:
\begin{align}\label{par_SM}
	s_\vep:=k^{-1}_\vep\circ\sigma_\vep\circ k_\vep: ~k^{-1}_\vep\big(\Omega^\vep_-(\Gamma_\vep)\big)\subset \La_0 \longrightarrow k^{-1}_\vep\big(\Omega^\vep_+(\Gamma_\vep)\big) \subset\La_0.
	%\text{Dom}(\cL^*)\subset \La_0\longrightarrow \La_0
\end{align}
Clearly, $s_\vep$ are still symplectic, and one can choose a common domain $\text{Dom}(\cL^*)$ for all
 $s_\vep$. Therefore,  combining formula \eqref{SMap} and Remark \ref{rem_par_der} we conclude the following result:

\begin{Pro}\label{Pro_SM_f_vep}
The $C^\ell$ parameterized scattering maps $s_\vep(I,\theta):$  $\text{Dom}(\cL^*)$ $\subset \Lambda_0\longrightarrow \Lambda_0$ given in \eqref{par_SM}, have the expansion:
\begin{equation}\label{new_scattering_map}
	s_\vep(I,\theta)=\Big(I+\vep \frac{\partial \cL^*}{\partial \theta}(I,\theta), ~\theta-\vep \frac{\partial \cL^*}{\partial I}(I,\theta)\Big)+O(\vep^2).
      \end{equation}
      where the $O$ symbol means estimates in the $C^{\ell -1} $ sense of the reminder. 
\end{Pro}

See  \cite{DLS2006, DLS2008} for more details. In view of Proposition \ref{Pro_SM_f_vep} we can formulate the following assumption:
\begin{itemize}
	\item [\bf(H3b)] let $\tau^*(I,\theta,0)$ be the $C^{r-1}$ function determined in \textbf{(H3a)}. Assume the reduced Poincar\'e function satisfies that $\mathbf{J}\nabla\cL^*(I,\theta)$ is transverse to the level set $\{I=I_0\}$ at some point $(I_0,\theta_0)\in \text{Dom}(\cL^*)$. That is equivalent to saying
\begin{equation}\label{non_degen_H3b}
	\frac{\partial \cL^*}{\partial\theta}(I_0,\theta_0)\neq 0.
\end{equation}
\end{itemize}

It is worth noting that assumption {\bf(H3b)} ensures that the vector field 
$\dot{x}=\mathbf{J}\nabla \cL^*(x)$ always has a trajectory along which the   action variables $I$ move a quantity independent of $\vep$.

It is remarkable that both assumptions {\bf(H3a)} and {\bf(H3b)}
amount to properties of the Melnikov potential  $L$
in \eqref{Def_PM_function}.  The gist of the genericity argument is
that, if they happen to fail for some $H_1$, a small modification of
the $H_1$ will make them true.  Both can be considered as transversality
properties on the functions. We also remark that the perturbations
needed to restore {\bf (H3a), (H3b) } are themselves rather arbitrary.
Hence, the assumptions can only fail for $H_1$ inside a manifold
of infinite codimension in the space of maps.  

\subsection{Geometric construction of the diffusing orbits}
In this paper, the construction of diffusing orbits is based on the geometric mechanism  in \cite{Gidea_Delallave_seara2020}.  This mechanism differs from earlier works, because it relies only  on the outer dynamics. There are almost no assumptions on the inner dynamics (only the Poincar\'e recurrence is needed), because its invariant objects (e.g., primary and secondary tori,  Aubry-Mather sets) on the NHIM are not used at all. The basic idea of this new mechanism is as follows: Assume that the Poincar\'e recurrence holds.
Given any pseudo-orbit, generated by the successive iterates of the scattering map $s_\vep$,  that moves $O(1)$ with respect to the perturbation, then there exists a true orbit moving a quantity of $O(1)$ for the original Hamiltonian system. 
Of course, if the inner dynamics in  the NHIM has no Poincar\'e recurrence,  then there exist diffusing orbits  determined just by the inner dynamics.

Through the parameterization $k_\vep$, the restriction of $f_\vep$ to $\La_\vep$ can be expressed on the same reference manifold $\La_0$. We use $\whf_\vep\big|_{\La_0}$ to denote this restriction map:  
 \[\whf_\vep\big|_{\La_0}= (k_\vep)^{-1}\circ f_\vep\big|_{\La_\vep}\circ k_\vep ~:~ \Lambda_0\to\Lambda_0.\]

\begin{The}\label{Thm_scattering_shadowing}
For $\vep\in (-\vep_0,\vep_0)\setminus\{0\}$ we have the  scattering map given in  \eqref{new_scattering_map}: 
\begin{equation*}
	s_\vep(I,\theta)=\Big(I+\vep \frac{\partial \cL^*}{\partial \theta}(I,\theta),~\theta-\vep \frac{\partial \cL^*}{\partial I}(I,\theta)\Big)+O(\vep^2).
\end{equation*}	
Assume that there is a point $x_0=(I_0,\theta_0)\in \text{Dom}(\cL^*)$ $\subset \Lambda_0$ and
$$\nabla \cL^*(I_0,\theta_0)\neq 0. $$ 
 Let $\gamma :$  $ [0,1]\to \Lambda_0$ be an integral curve through $x_0$ induced by the Hamiltonian vector field 
$\dot{x}=\mathbf{J}\nabla \cL^*(x)$. Suppose  that there is a neighborhood 
$V \subset \text{Dom}(\cL^*)$ of $\gamma([0,1])$ such that almost every point in $V$ is recurrent for $\whf_\vep\big|_{\La_0}$. 
Denote by $\gamma_\vep:=k_\vep\circ\gamma$  the corresponding curve in $\Lambda_\vep$.

 Then there exist a small $\vep_1=\vep_1(H_1)\in(0,\vep_0)$ and a constant $K=K(H_1)>0$, such that for every  $0<|\vep|<\vep_1$ and every $\delta>0$, there is an orbit $\{z_i\}_{i=0,\cdots,m}$ of the diffeomorphism $f_\vep$ where $m=O(1/\vep)$ and
 \[
z_{i+1}=f_\vep ^{k_i}(z_i),\quad \text{~for some~} k_i\in\Z^+,
\]
and 
\[
\textup{dist}(z_i,\gamma_\vep(t_i))<\delta+K\vep, 
\]
where $t_i=i\cdot \vep$ and $0=t_0<t_1<\cdots<t_m\leq 1$.
\end{The}

 We provide a sketch of the proof in Appendix \ref{Appendix_section} for the reader's convenience. One can also refer to  \cite[Theorem  3.11]{Gidea_Delallave_seara2020} for a complete proof.  

Since  the recurrence assumption is satisfied automatically in our Hamiltonian model
if there are no unbounded orbits 
in the manifold, under hypotheses {\bf (H3a), (H3b)}, 
that either there are unbounded orbits in the NHIM 
or that there are sequences of homoclinic excursions that follow 
the sequence. In either  of the two sides of the alternative, there is
diffusion. 

\begin{The}\label{existence_orbits}
Suppose that  $\HH_\vep=H_0+\vep H_1$ satisfies conditions \textrm{\bf(H1), (H2), (H3a)} and \textrm{\bf(H3b)}, and $I_0$ is the point given in \eqref{non_degen_H3b}. Then 
there exist $\vep_1=\vep_1(H_1)>0$ and $\rho=\rho(H_1)>0$, such that for any $0<|\vep|<\vep_1$ and any $\delta>0$, the Hamiltonian flow  admits a trajectory $\left(p(t), q(t), I(t), \varphi(t)\right)$  whose $I$-component satisfy 
\[\sup_{t>0}\|I(t)-I(0)\|\geq \rho,\]
and  $\|I(0)-I_0\|<\delta+K\vep.$
The constant $K=K(H_1)$ does not depend on $\vep$.
\end{The}

The proof is postponed to Appendix \ref{Appendix_section}, which comes mainly from  \cite{Gidea_Delallave_seara2020}. 
Finally, we also refer to \cite{Gidea_Delallave_seara2020_skip} for the result using accessibility and several scattering maps.

\section{Proof of main result}\label{section_proof}
This section aims to prove our main theorems given in Section \ref{sec_mainresult}. Thanks to Theorem \ref{existence_orbits}, it is sufficient to show that the non-degeneracy assumptions \textbf{(H3a)}--\textbf{(H3b)}  are  generic properties in the analytic category. Note  that  assumption \textbf{(H3b)}  depends on the choice of the function $\tau^*$ given in assumption \textbf{(H3a)}.

Now, we first  verify that \textbf{(H3a)}  is a generic property in the analytic category. 
According to \eqref{Def_PM_function}, the Melnikov potential is 
\begin{equation}\label{intL0}
\begin{split}
L(\tau, I,\vph,s)=-\int^\infty_{-\infty}\Big[ H_1(p^0(\tau+t\bar{1}),q^0(\tau+t\bar{1}),I,\vph+\omega(I)t, s+t) -H_1(0,0,I,\vph+\omega(I)t, s+t)\Big]\,dt.
\end{split}
\end{equation}
where $p^0(\tau+t\bar{1})=\big(p_1^0(\tau_1+t),\cdots,p_n^0(\tau_n+t)\big)$ and $
q^0(\tau+t\bar{1})=\big(q_1^0(\tau_1+t),\cdots,q_n^0(\tau_n+t)\big)$.
\begin{The}\label{H3a_genericity}
Let $\kappa>0$.
For each point $(I_0,\vph_0,s_0)\in \cB_*\times\T^{d}\times\T$, the set of perturbations $H_1$  satisfying the following non-degeneracy property  \textrm{\bf{(R)}}, is  $C^\omega_\kappa$ dense and $C^3$ open (which also implies $C^\omega_\kappa$-openness). 
\begin{itemize}[\bf(R)]
	\item there exists an open set $U^-:=\mathcal{I}\times\mathcal{J}$ containing the point $(I_0,\vph_0,s_0)$ with $\mathcal{I}$ a ball in $\cB_*$ and $\mathcal{J}$ an open set in $\T^{d}\times\T$, such that for each point $(I,\vph,s)\in U^-$ the map
$$\tau\in\R^n\longmapsto L(\tau,I,\vph,s)$$
 has a non-degenerate critical point $\tau^*$, i.e., 
$$\frac{\partial L}{\partial \tau}(\tau^*,I,\vph,s)=0, \quad \textup{det}\left[\frac{\partial^2 L}{\partial \tau_i\tau_j}(\tau^*,I,\vph,s)\right]_{1\leq i,j\leq n}\neq 0.$$
By the implicit function theorem, $\tau^*$ is locally given by   
$$\tau^*=\tau^*(I,\vph,s).$$ 
\end{itemize}  
\end{The}

Before proving it, we need some lemmas.

\begin{Lem}\label{lem_nonzero}
	For each $\kappa>0$ and each $a\in\R\setminus\{0\}$, there is a sequence of real analytic functions $f_l\in C^\omega_\kappa(\T)$, $l=1,\cdots,n$, satisfying 
	\begin{equation*}
		\int_{-\infty}^\infty\big[ f_l(q_l^0(t))-f_l(0)\big]\cdot e^{iat}\,dt\neq 0,
	\end{equation*}
where $e^{iat}=\cos at+i\sin at$, and $q_l^0(t)$ is the $q_l$-coordinate of the unperturbed homoclinic orbit $(p^0, q^0)$. 
\end{Lem}
\begin{proof}
	For simplicity we only verify it for $l=1$ and the others are similar.  
It is sufficient to prove that there exists a  function  
$f_1$ in the set $X:=$ $\{\cos kx, \sin kx ~:~ k\in\Z\}$, such that 
	\begin{equation*}
		\int_{-\infty}^\infty\big[ f_1(q_1^0(t))-f_1(0)\big]\cdot e^{iat}\,dt\neq 0.
	\end{equation*}

Suppose that for all $f\in X$,  
$\int_{-\infty}^\infty[ f(q_1^0(t))-f(0)]\cdot e^{iat}\,dt$ $= 0,$ then using  the theory of Fourier analysis it is not difficult to prove that \begin{equation}\label{zero_forall_f}
		\int_{-\infty}^\infty\big[ \phi(q_1^0(t))-\phi(0)\big]\cdot e^{iat}\,dt= 0,\quad \text{~for each~} \phi\in C^1(\T).
	\end{equation}

On the other hand, recalling that $q_1^0(t)$ converges exponentially to $0$ as $t$ tends to $\pm\infty$, we can find a small closed interval  $J\subset \T\setminus\{0\}$, such that the orbit $q^0_1(t)$ passes through $J$ \emph{when and only when}  $t\in [\overline{t}-\sigma, \overline{t}+\sigma]$ for some $\overline{t}\in\R,  \sigma>0$. 
By narrowing the interval $J$ if necessary, we can let $\sigma <\frac{\pi}{4a}$. 

Let us pick a non-negative  function $h\in C^1(\T)$ satisfying  $h(0)=0$ and the support $\text{supp} h=J$, then  
\begin{equation*}
		\int_{-\infty}^\infty\big[ h(q_1^0(t))-h(0)\big]\cdot e^{iat}\,dt=\int_{\overline{t}-\sigma}^{\overline{t}+\sigma} h(q_1^0(t))\cdot e^{iat}\,dt=e^{ia\overline{t}}\int_{-\sigma}^{\sigma} h(q_1^0(t+\overline{t}))\cdot e^{iat}\,dt\neq 0,
	\end{equation*}
since $ h(q_1^0(t+\overline{t}))$ and $\cos at$ are both positive when $t\in(-\sigma,\sigma)$. This  contradicts  \eqref{zero_forall_f}. Now the lemma follows immediately from what we have proved. 
\end{proof}

\begin{Lem}\label{Lem_inf_0}
	For every point $(I,\vph,s)\in \cB_*\times\T^d\times\T$, the Melnikov potential $L(\tau,I,\vph,s)$ satisfies
	\begin{equation}\label{inf_0}
	\inf\limits_{\tau\in\R^n}\left\|\frac{\partial L}{\partial\tau}(\tau,I,\vph,s)\right\|=0,
\end{equation}
where the norm $\left\|\frac{\partial L}{\partial\tau}\right\|:=\sum_{i=1}^n \left|\frac{\partial L}{\partial\tau_i}\right|$.
\end{Lem}
\begin{proof}
Fix a point $(\hat{I},\hat{\vph},\hat{s})\in \cB_*\times\T^d\times\T$. If $\frac{\partial L}{\partial \tau}(\tau,\hat{I},\hat{\vph},\hat{s})=0$ is attained at some point $\tau\in\R^n$, then we have finished the proof.

For the case where $\frac{\partial L}{\partial \tau}(\tau,\hat{I},\hat{\vph},\hat{s})$ is always non-zero,   we assume by contradiction that there is  $\delta>0$ such that 
\begin{equation}\label{contradiction_delta}
	\left\|\frac{\partial L}{\partial\tau}(\tau,\hat{I},\hat{\vph},\hat{s})\right\|\geq \delta,\quad \text{for all~}\tau\in\R^n.
\end{equation}
Observe that the homoclinic orbit $(p^0,q^0)$ is contained in a compact and convex set $D\subset\R^n\times\T^n$, so we can define 
\[M(\hat{I}):= \max_{i=1,\cdots,n }\left\{\sup_{\substack{ (p,q)\in D,~ (\vph,s)\in\T^{d}\times\T}}\left| \frac{\partial H_1}{\partial p_i}(p,q,\hat{I},\vph,s)  \right|+\left| \frac{\partial H_1}{\partial q_i}(p,q,\hat{I},\vph,s)  \right|\right\},
\]
which is finite and bounded. Using the mean value theorem to \eqref{intL0}, 	for every $\tau\in\R^n$ we have
\begin{equation}\label{L_bounded}
	\begin{split}	
\big|L(\tau, \hat{I},\hat{\vph},\hat{s})\big|\leq & \sum_{i=1}^n \int^{\infty}_{-\infty}M(\hat{I})\cdot
\big|p^0_i(t+\tau_i)\big|+M(\hat{I})\cdot
\big|q^0_i(t+\tau_i)\big|\,dt\\
= &  M(\hat{I})\cdot \sum_{i=1}^n \int^{\infty}_{-\infty}
\big|p^0_i(t)\big|+
\big|q^0_i(t)\big|\,dt\leq  CM(\hat{I}),
	\end{split}
\end{equation}
where $C$ is a constant, and the last inequality is a consequence of the fact that $(p_i^0(t),q_i^0(t))$  converges exponentially to $(0,0)$ as $t\to \pm\infty$. Similarly, as the derivative $(\dot{p}_i^0(t),\dot{q}_i^0(t))$ also converges exponentially to $(0,0)$ as $t\to \pm\infty$,	
we deduce that
\begin{equation}\label{partial_L_bounded}
	\begin{split}
\left\|\frac{\partial L}{\partial\tau}(\tau,\hat{I},\hat{\vph},\hat{s})\right\|=&
\sum_{i=1}^n\left|\frac{\partial L}{\partial \tau_i}(\tau, \hat{I},\hat{\vph},\hat{s})\right|\\\leq & M(\hat{I})\cdot\sum_{i=1}^n \int^{\infty}_{-\infty} |\dot{p}^0_i(t+\tau_i)|+|\dot{q}^0_i(t+\tau_i)|\,dt\\
= & M(\hat{I})\cdot\sum_{i=1}^n \int^{\infty}_{-\infty} |\dot{p}^0_i(t)|+|\dot{q}^0_i(t)|\,dt\leq  C'M(\hat{I}),	
	\end{split}
\end{equation}
where $C'>0$ is a constant.

Recall that $(\hat{I},\hat{\vph},\hat{s})$ is fixed, then we consider an auxiliary differential equation:
\[
\dot{x}= \frac{\partial L}{\partial\tau}(x,\hat{I},\hat{\vph},\hat{s}), \quad x\in\R^n.
\]
From \eqref{partial_L_bounded} we see that the vector field above is bounded, which implies the flow is complete, that is all solutions are well defined for $t\in\R$. We pick one solution $x(t):\R\to\R^n$ and consider the one-dimensional function $t\longmapsto L(x(t),\hat{I},\hat{\vph},\hat{s})$. Then, for any $T>0$ we obtain 
\begin{align}\label{LT_est}
\begin{aligned}
	L(x(T),\hat{I},\hat{\vph},\hat{s})-L(x(0),\hat{I},\hat{\vph},\hat{s})=&\int_{0}^T\sum_{i=1}^n\left|\frac{\partial L}{\partial \tau_i}(x(t), \hat{I},\hat{\vph},\hat{s})\right|^2\, dt\\
	\geq & \frac{1}{n}\int_{0}^T \left\|\frac{\partial L}{\partial\tau}(x(t),\hat{I},\hat{\vph},\hat{s})\right\|^2\,dt\\
	\geq & \frac{\delta^2}{n}T,
\end{aligned}
\end{align}
Here, we have used \eqref{contradiction_delta} in the last inequality.
As $T$ can be arbitrarily large, the estimate \eqref{LT_est} yields 
$\sup_{\tau\in\R^n}|L(\tau,\hat{I},\hat{\vph},\hat{s})|=+\infty$.
This contradicts the boundedness estimate \eqref{L_bounded}. Therefore, we finish the proof of \eqref{inf_0}.
\end{proof}

\begin{Rem}
The estimates \eqref{L_bounded}--\eqref{partial_L_bounded} also imply that the Melnikov potential  $L$ and its partial derivative  $\partial L/\partial\tau$ are uniformly bounded in $(\tau,\vph,s)$, that is for each $I$ there is a constant $C(I)>0$ such that
\[ |L(\tau,I,\vph,s)|\leq C(I),\quad \left\|\frac{\partial L}{\partial\tau}(\tau,I,\vph,s)\right\|\leq C(I),\quad \text{for all~} (\tau,\vph,s)\in\R^n\times\T^{d}\times\T.\]  	
\end{Rem}
\begin{Rem}[\textbf{Existence of critical points}]
	We stress that in the case when $n=1$, the identity $L(\tau,I,\vph,s)=L(0,I,\vph-\omega(I)\tau, s-\tau)$ holds, so the one-dimensional map $\tau\longmapsto L(\tau,I_0,\vph_0,s_0)$ always has critical points for each fixed $(I_0,\vph_0,s_0)$. 
	See \cite{DLS2006}. 
	
In the case when  $n>1$,  for each $I_0$ there exists $(\vph_0,s_0)$ such that
 the map $\tau\in\R^n\longmapsto L(\tau,I_0,\vph_0,s_0)$  has critical points. See \cite{Gidea_delaLlave2018}.
\end{Rem}

Now, we proceed to prove Theorem \ref{H3a_genericity}. 

\begin{proof}[Proof of Theorem \ref{H3a_genericity}]
The openness is evident. In fact, a non-degenerate
critical point for the map $\tau\mapsto L(\tau,I,\varphi,s)$ remains non-degenerate after a sufficiently small $C^2$ perturbation, which therefore gives $C^3$-openness and also $C^\omega_\kappa$-openness. Here, $C^3$-smoothness is necessary because $H_1$ needs to satisfy the lowest regularity (i.e. $C^3$) of the unperturbed system $H_0$  so that all the results obtained in Section \ref{section_mechanism} are still valid.

Then it remains to show that the existence of non-degenerate critical points is a dense property in the $C^\omega_\kappa$ topology. 
The proof splits into two steps.

\textbf{Step 1:} Fix the point $(I_0,\vph_0,s_0)$. If the map $\tau\longmapsto L(\tau,I_0,\vph_0,s_0)$ has no critical points, 
 we will show that there is an arbitrarily  small perturbation to $H_1$ to create critical points. For this purpose, we take a small number $\delta>0$ and add a small perturbation  $\delta_2 H_2$ to $H_1$ with $\delta_2\in(0,\delta/2)$, and $H_2$ is of the form
$$H_2(q,t)=\sum_{i=1}^n f_i(q_i)\cdot\cos (t+b_i)\in C^\omega_\kappa
,$$ 
the   coefficients $\{b_i\}_{i=1}^n$ and the analytic functions $\{f_i\in C^\omega_\kappa(\T)\}_{i=1}^n$ will be determined later. By multiplying a constant if necessary, we can always let $\|H_2\|_\kappa<1$.
Hence the new Melnikov potential, denoted by $L^{\delta_2}$, associated to the  Hamiltonian $H_0+\vep (H_1+ \delta_2 H_2)$ is
\begin{align}\label{Ldelta2}
\begin{aligned}
	L^{\delta_2}(\tau,I,\vph,s)&=L(\tau,I,\vph,s)-\delta_2 \sum_{i=1}^n\int^{\infty}_{-\infty}\Big(f_i(q^0_i(t+\tau_i))-f_i(0)\Big)\cdot\cos(s+t+b_i)\, dt\\
	&=L(\tau,I,\vph,s)-\delta_2 \sum_{i=1}^n\int^{\infty}_{-\infty}\Big(f_i(q^0_i(t))-f_i(0)\Big)\cdot\cos(s-\tau_i+b_i+t)\, dt\\
	&=L(\tau,I,\vph,s)+\delta_2 \sum_{i=1}^n\Big[A_{i,1}\cdot\cos(s-\tau_i+b_i)- A_{i,2}\cdot\sin (s-\tau_i+b_i)\Big]\\
	&=L(\tau,I,\vph,s)+\delta_2\sum_{i=1}^n\sqrt{A^2_{i,1}+A_{i,2}^2}\cdot\cos(s-\tau_i+b_i+\alpha_i),	
\end{aligned}
\end{align}
where for each $i=1,\cdots, n$, the constants $A_{i,1},A_{i,2}$ are given by
\[A_{i,1}=-\int^\infty_{-\infty}\Big(f_i(q_i^0(t))-f_i(0)\Big)\cos t \,dt,\quad A_{i,2}=-\int^\infty_{-\infty}\Big(f_i(q_i^0(t))-f_i(0)\Big)\sin t \,dt
\]
and the angle $\alpha_i=\arccos(A_{i,1}/\sqrt{A^2_{i,1}+A_{i,2}^2})$. We can ensure each $A^2_{i,1}+A_{i,2}^2\neq 0$ by suitably choosing $f_i$ (see Lemma \ref{lem_nonzero}). This gives rise to
\begin{equation*}\label{par_der_1}
\frac{\partial L^{\delta_2}}{\partial \tau}(\tau,I,\vph,s)=(x_1,\cdots,x_n)^\intercal
\end{equation*}
with 
\begin{equation*}\label{underterm_b}
x_i=x_i(\tau,I,\vph,s)=
\frac{\partial L}{\partial \tau_i}(\tau,I,\vph,s)+\delta_2\sqrt{A^2_{i,1}+A_{i,2}^2}\cdot\sin(s-\tau_i+b_i+\alpha_i).	
\end{equation*}
In particular, for the point $(I_0,\vph_0,s_0)$ we can invoke Lemma \ref{Lem_inf_0} to find a point $\tau^*$ satisfying
\begin{align*}
	\left\|\frac{\partial L}{\partial\tau}(\tau^*,I_0,\vph_0,s_0)\right\|<\min_{1\leq i\leq n}~\delta_2\sqrt{A^2_{i,1}+A_{i,2}^2}.
\end{align*}
Then for each $i$ we can find $b_i\in[0,2\pi]$ such that $x_i(\tau^*,I_0,\vph_0,s_0)=0$, which therefore yields
\[
\frac{\partial L^{\delta_2}}{\partial \tau}(\tau^*,I_0,\vph_0,s_0)=0.
\]

\textbf{Step 2:} We have already shown the existence of critical points is a dense property. In the above argument, we may pick  $H_2=0$ whenever the map $\tau\to L(\tau,I_0,\vph_0,s_0)$ already has critical points.

Now, let us turn to check the non-degeneracy of the critical points.
If $(\tau^*,I_0,\vph_0,s_0)$ is a degenerate critical point of $L^{\delta_2}$, i.e.,
\begin{align}\label{L_degeneracy}
\frac{\partial  L^{\delta_2}}{\partial \tau}(\tau^*,I_0,\vph_0,s_0)=0, \quad \text{det}\left[\frac{\partial^2  L^{\delta_2}}{\partial \tau_i\tau_j}(\tau^*,I_0,\vph_0,s_0)\right]_{1\leq i,j\leq n}= 0.
\end{align}
then we have to  continue to add a small perturbation to create non-degeneracy.
Indeed, we may pick a perturbation $\delta_3 H_3$ to $H_1+\delta_2H_2$ where  $\delta_3\in(0,\delta/2)$ and
\[
\delta_3 H_3(q,t)=\delta_3 \sum_{i=1}^n   g_i( q_i) \cdot\cos (t+c_i).
\]
 The value $\delta_3$, the coefficients $\{ c_i\in\R\}_{i=1}^n$ and the analytic functions $\{g_i\in C^\omega_\kappa(\T)\}_{i=1}^n$ will be determined later. Without loss of generality we let $\|H_3\|_\kappa <1$. 
Using arguments analogous to \eqref{Ldelta2},
 the new Melnikov potential, denoted by $L^{\delta_2,\delta_3}$, associated to the  Hamiltonian $H_0+\vep (H_1+ \delta_2 H_2+\delta_3 H_3)$ is
\begin{equation}
\begin{split}
L^{\delta_2,\delta_3}(\tau,I,\vph,s)&=L^{\delta_2}(\tau,I,\vph,s)+\delta_3\sum_{i=1}^n \Big[ B_{i,1} \cdot\cos(s-\tau_i+c_i)-B_{i,2}\cdot\sin(s-\tau_i+c_i)\Big],\\
&=L^{\delta_2}(\tau,I,\vph,s)+\delta_3 \sum_{i=1}^n\sqrt{B_{i,1}^2+B_{i,2}^2}\cdot\cos(s-\tau_i+c_i+\beta_i),
\end{split}
\end{equation}
where the constants $B_{i,1}, B_{i,2}$ are
\[B_{i,1}=-\int^\infty_{-\infty}\Big(g_i(q_i^0(t))-g_i(0)\Big)\cos t \,dt,\quad 
B_{i,2}=-\int^\infty_{-\infty}\Big(g_i(q_i^0(t))-g_i(0)\Big)\sin t \,dt
\]
and $\beta_i=\arccos(B_{i,1}/\sqrt{B_{i,1}^2+B_{i,2}^2})$. Here, we can ensure each $B^2_{i,1}+B^2_{i,2}\neq 0$ by suitably choosing $g_i$, see Lemma \ref{lem_nonzero}.

For each $i=1,\cdots,n$  we  take
\begin{equation}\label{cci}
	c_i:=-s_0+\tau^*_i-\beta_i,
\end{equation}
then  the Hessian matrix of the map $\tau\longmapsto L^{\delta_2,\delta_3}(\tau,I_0,\vph_0,s_0)$  is
\[
\left[
  \frac{\partial^2 L^{\delta_2,\delta_3}}{\partial \tau_i\tau_j}(\tau^*,I_0,\vph_0,s_0)  
\right]_{1\leq i,j\leq n}
= \left[\frac{\partial^2 L^{\delta_2}}{\partial \tau_i\tau_j}(\tau^*,I_0,\vph_0,s_0)\right]_{1\leq i,j\leq n}+  
\left[
\begin{array}{cccc}
\delta_3 \lambda_1 &     &       &    \\
  & \delta_3\lambda_2  &       &     \\
  &     &\ddots &    \\
  &     &       & \delta_3 \lambda_n  
\end{array}\right]
\]
where the second term on the right-hand side is a diagonal matrix, and as a result of  \eqref{cci}, 
\begin{equation}\label{yppt}
	\lambda_i=-\sqrt{B_{i,1}^2+B_{i,2}^2}\cdot\cos(s_0-\tau^*_i+c_i+\beta_i)=- \sqrt{B_{i,1}^2+B_{i,2}^2}\neq 0.
\end{equation}

Denoting 
\[
v(\delta_3):=\text{det}\left[\frac{\partial^2 L^{\delta_2,\delta_3}}{\partial \tau_i\tau_j}(\tau^*,I_0,\vph_0,s_0)\right]_{1\leq i,j\leq n}
\]
In particular,  $v(0)=0$ as a consequence of \eqref{L_degeneracy}. It is not difficult to check that the one-dimensional function $v:\delta_3\to\R$ is a polynomial function of degree $n$,  and the leading term of $v$ is
$\bigg(\prod_{i=1}^n\lambda_i \bigg)\delta_3^n.$ Thanks to \eqref{yppt}, the leading coefficient is  non-zero, which
 implies that the polynomial  $v$  has at most $n$ zeros. Consequently, we can choose arbitrarily small $\delta_3>0$ such that 
 \[v(\delta_3)\neq 0.\]

In conclusion, we have constructed a perturbation $\delta_2H_2+\delta_3 H_3 $ to $H_1$ where
 \[\|\delta_2H_2+\delta_3 H_3\|_\kappa\leq \delta_2\|H_2\|_\kappa+\delta_3 \|H_3\|_\kappa\leq\delta_2+\delta_3<\delta.\] 
As $\delta>0$ can be arbitrarily small,  the existence of non-degenerate critical points for $\tau\mapsto L(\tau,I_0,\vph_0,s_0)$ is a dense property in the $C^\omega_\kappa$ topology. Finally,  using the implicit function theorem we can obtain  an open neighborhood  $U^-=\mathcal{I}\times\mathcal{J}\subset\cB_*\times\T^{d+1}$ of the point $(I_0,\vph_0,s_0)$, such that for each $(I,\vph,s)\in U^-$ the map 
$$\tau\in\R^n\longmapsto L(\tau,I,\vph,s)$$
has a non-degenerate critical point $\tau^*=\tau^*(I,\vph,s).$ This finishes our proof.
\end{proof}

We have provided a constructive   proof for Theorem \ref{H3a_genericity}.  
Next, we proceed to show the genericity of assumption {\bf(H3b)}.

\begin{The}\label{H3b_genericity}
Given $\kappa>0$.
	The set of perturbations $H_1$ satisfying the non-degeneracy assumption {\bf(H3b)} is $C^\omega_\kappa$ dense and $C^3$ open (which therefore implies $C^\omega_\kappa$-openness).
\end{The}
\begin{proof}
	Since  assumption {\bf(H3a)} has already been proved to be open and dense, in the following proof we restrict our discussions to the case where  assumption {\bf(H3a)} always holds. As mentioned previously, $\tau^*(I,\theta,0)\in\R^n$ denotes the non-degenerate critical point for the map $\tau\longmapsto L(\tau, I,\theta, 0)$. Just like  assumption {\bf(H3a)}, it is easy to see that the non-degeneracy assumption {\bf(H3b)} is open in the $C^3$ topology, which directly gives $C^\omega_\kappa$-openness. Thus, the only thing left is to verify the density of {\bf(H3b)}.
	
In what follows, we fix a point $(\hat{I},\hat{\theta})\in \text{Dom}(\cL^*)$.
If $\partial\cL^*/\partial\theta(\hat{I},\hat{\theta})\neq 0$, then we have finished.
 
 Suppose now
\begin{equation}\label{contr_L_theta}
	\frac{\partial \cL^*}{\partial\theta}(\hat{I},\hat{\theta})=0,
\end{equation} 
we will add a small perturbation to create non-degeneracy. More precisely, 
we add a perturbation $\delta H_2$ to $H_1$ where  the number $\delta>0$ is small enough and the analytic function $H_2$ will be determined later. Then the new Melnikov potential is
\[L^\delta=L+\delta \widetilde{L}\] where
\begin{equation}\label{L_H2}
\begin{split}
\widetilde{L}(\tau, I,\vph,s)=-\int^\infty_{-\infty}\Big[& H_2(p^0(\tau+t\bar{1}),q^0(\tau+t\bar{1}),I,\vph+\omega(I)t, s+t)-H_2(0,0,I,\vph+\omega(I)t, s+t)\Big]\,dt.	
\end{split}	
\end{equation}
As {\bf(H3a)} holds for $L$, it follows from the implicit function theorem that near $\tau^*$ there is a unique and non-degenerate critical point $\tau^{*,\delta}=\tau^{*,\delta}(I,\theta,0)$  for the 
map $\tau\longmapsto$ $ L^\delta(\tau,I,\theta,0)$. This critical point has an expansion in $\delta$ as follows,
$$\tau^{*,\delta}(I,\theta,0)=\tau^*(I,\theta,0)+\delta\,\widetilde{\tau}^*(I,\theta,0)+O(\delta^2).$$
Since $\tau^{*,\delta}(I,\theta,0)$ solves the equation $\partial L^\delta/\partial\tau(\tau^{*,\delta},I,\theta,0)=0$,  we obtain
\[
\frac{\partial L}{\partial\tau}(\tau^{*},I,\theta,0)+\delta\Big[\frac{\partial^2L}{\partial \tau^2} (\tau^*,I,\theta,0)\widetilde{\tau}^*+\frac{\partial \widetilde{L}}{\partial\tau}(\tau^*,I,\theta,0)\Big]+O(\delta^2)=0
\]
for all small $\delta$. Then,
owing to $\partial L/\partial\tau(\tau^{*},I,\theta,0)=0$ we obtain
\begin{equation*}
	\widetilde{\tau}^*(I,\theta,0)=-\bigg (\frac{\partial^2L}{\partial \tau^2}(\tau^*,I,\theta,0)\bigg)^{-1}\frac{\partial \widetilde{L}}{\partial\tau}(\tau^*,I,\theta,0),
\end{equation*}
where the matrix $\partial^2L/\partial \tau^2 (\tau^*,I,\theta,0)$ is invertible as a result of \textbf{(H3a)}. Hence the new reduced Poincar\'e function $(\cL^\delta)^*$ has an expansion in $\delta$ as follows:
\begin{equation}\label{L_delta_expression}
	\begin{split}
	(\cL^\delta)^*(I,\theta)=&L^{\delta}(\tau^{*,\delta},I,\theta,0)	\\=&L(\tau^*,I,\theta,0)+\delta\Big[\frac{\partial L}{\partial \tau}(\tau^*,I,\theta,0)\,\widetilde{\tau}^*
	+\widetilde{L}(\tau^*,I,\theta,0)
	\Big]+O(\delta^2)	\\
 = &\cL^*(I,\theta)
	+\delta\,\widetilde{L}(\tau^*,I,\theta,0)+O(\delta^2).
	\end{split}
\end{equation}
Here, the last equality follows from the definition of $\tau^*$ at which $\partial L/\partial\tau(\tau^*,I,\theta,0)= 0$.

We want to show that $\frac{\partial (\cL^\delta)^*}{\partial\theta}$ is non-zero at the point $(\hat{I}, \hat{\theta})$. In view of \eqref{contr_L_theta} and \eqref{L_delta_expression}, it suffices to prove $$\frac{\partial \widetilde{L}}{\partial\theta}(\tau^*,\hat{I},\hat{\theta},0)\neq 0.$$ 
To achieve this,  we choose $H_2$  as follows
\begin{equation}\label{H2form}
H_2(q,\vph)=F(q_1)\sum_{i}^d \cos(\vph_i+c_i),	
\end{equation}
where the one-dimensional function $F\in C^\omega_\kappa(\T)$ and the sequence of numbers $\{c_i\}_{i=1}^d$ will be determined later. Then we deduce from \eqref{L_H2} that
\begin{equation*}
\begin{split}
	\widetilde{L}(\tau,I,\vph,s)=\sum_{i=1}^d A_i(I,\tau)\cdot\cos(\vph_i+c_i)-B_i(I,\tau)\cdot\sin(\vph_i+c_i),
\end{split}	
\end{equation*}
where the coefficients 
\[A_i(I,\tau)=-\int^\infty_{-\infty}\left(F(q_1^0(\tau_1+t))-F(0)\right)\cdot\cos \omega_i(I)t \,dt=-\int^\infty_{-\infty}\left(F(q_1^0(t))-F(0)\right)\cdot\cos \big(\omega_i(I)\cdot(t-\tau_1)\big) \,dt,\]
\[B_i(I,\tau)=-\int^\infty_{-\infty}\left(F(q_1^0(\tau_1+t))-F(0)\right)\cdot\sin\omega_i(I)t \,dt=-\int^\infty_{-\infty}\left(F(q_1^0(t))-F(0)\right)\cdot\sin\big(\omega_i(I)\cdot(t-\tau_1)\big)\,dt.\]
Note that $A_i(I,\tau)$ and $B_i(I,\tau)$ do not depend on $\tau_2,\cdots,\tau_n$.
In particular, for  $\tau^*=\tau^*(I,\theta,0)$ we have
\begin{equation*}
\begin{split}
	\widetilde{L}(\tau^*,I,\theta,0)&=\sum_{i=1}^d A_i(I,\tau^*)\cdot\cos(\theta_i+c_i)-B_i(I,\tau^*)\cdot\sin(\theta_i+c_i)\\
	&=\sum_{i=1}^d C_i(I,\tau^*)\cdot\cos\Big(\theta_i+c_i+\alpha_i(I,\tau^*)\Big),
\end{split}	
\end{equation*}
where $C_i(I,\tau^*)=\sqrt{A_i^2(I,\tau^*)+B^2_i(I,\tau^*)}$. The angle $\alpha_i(I,\tau^*)=\arccos \left(A_i(I,\tau^*)/C_i(I,\tau^*)\right)$ if the value  $C_i(I,\tau^*)\neq 0$. We choose $\alpha_i(I,\tau^*)=0$ whenever $C_i(I,\tau^*)=0$.

For the fixed point $(\hat{I},\hat{\theta})$, we can invoke Lemma \ref{lem_nonzero} to  choose a one-dimensional function $F(q_1)$ such that 
\[C_1(\hat{I},\tau^*)=\sqrt{A_1^2(\hat{I},\tau^*)+B^2_1(\hat{I},\tau^*)}\neq 0, \qquad\text{where~}\tau^*=\tau^*(\hat{I},\hat{\theta},0).\]
As we will see below, the values of $C_i(\hat{I},\tau^*)$ for $i=2,3,\cdots,d$ play no role in the following proof.

Observe that 
 \begin{equation*}
\frac{\partial \widetilde{L}}{\partial\theta}(\tau^*,\hat{I},\hat{\theta},0)=(y_1,\cdots,y_d)^\intercal,
\end{equation*}
 where
 \begin{equation*}
y_i=-C_i(\hat{I},\tau^*)\cdot\sin\Big(\hat{\theta}_i+c_i+\alpha_i(\hat{I},\tau^*)\Big),\qquad i=1,\cdots,d.
\end{equation*}
Now we can choose
\[c_1=-\hat{\theta}_1-\alpha_1(\hat{I},\tau^*)+\frac{\pi}{2},\quad\text{and}\quad c_i=-\hat{\theta}_i-\alpha_i(\hat{I},\tau^*)\quad\text{for~} i=2,\cdots,d\]
which therefore gives
\begin{equation}\label{non_degen_direction}
	\frac{\partial \widetilde{L}}{\partial\theta}(\tau^*,\hat{I},\hat{\theta},0)=(-C_1(\hat{I},\tau^*), ~0, \cdots, 0)^\intercal\neq 0.
\end{equation}
Together with \eqref{contr_L_theta} and \eqref{L_delta_expression}, this implies 
\[\frac{\partial (\cL^\delta)^*}{\partial\theta}(\hat{I},\hat{\theta})=\delta\, \frac{\partial \widetilde{L}}{\partial\theta}(\tau^*,\hat{I},\hat{\theta},0)+O(\delta^2)\neq 0\]
 as long as $\delta>0$ is sufficiently small. This completes the proof.
\end{proof}

\begin{Rem}\label{anypoint_open_set2}
In fact, from Theorem \ref{H3a_genericity} and the proof of  Theorem \ref{H3b_genericity} one can find that for any given point $(I_0,\theta_0)$, the set of perturbations $H_1$ satisfying  $(I_0,\theta_0)\in \text{Dom}(\cL^*)$ and
\begin{equation*}
\frac{\partial \cL^*}{\partial\theta}(I_0,\theta_0)\neq 0.
\end{equation*}
 is $C^\omega_\kappa$ open and dense.
\end{Rem}

Therefore, Theorem \ref{H3a_genericity} together with Theorem \ref{H3b_genericity} and Remark \ref{anypoint_open_set2}  leads to the following result.

\begin{The}\label{comb_H3ab}
Let $\kappa>0$ and $I_0\in\cB_*$. Then the set of all perturbations  $H_1$, satisfying  assumption \textrm{\bf{(H3a)}} with $\{I=I_0\}\cap U^-\neq \emptyset$ and  assumption \textrm{\bf{(H3b)}} with 
$\frac{\partial \cL^*}{\partial\theta}(I_0,\theta_0)\neq 0$ for some $\theta_0\in\T^d$, is  dense and open in the $C^\omega_\kappa$ topology. Actually, it is also $C^3$ open.	
\end{The}

Now, we are ready to prove our main result in Section \ref{sec_mainresult}.

\begin{proof}[Proof of Theorem \ref{Main_Thm1}]
Let $\kappa>0$ and the point $I_0$ be fixed, we can invoke the above Theorem \ref{comb_H3ab} to find a $C^\omega_\kappa$  open and dense set $\cU$, such that for each $H_1\in\cU$ the Hamiltonian $\HH_\vep=H_0+\vep H_1$ satisfies the non-degeneracy assumptions  {\bf (H3a)--\bf(H3b)}. In particular, the reduced Poincar\'e function $\cL^*$ satisfies  $\frac{\partial \cL^*}{\partial\theta}(I_0,\theta_0)\neq 0$ for some $\theta_0\in\T^d$. 

On the other hand, as we have already assumed that $H_0$ satisfies conditions {\bf (H1)--\bf(H2)},  it follows from Theorem \ref{existence_orbits} that there exist $\vep_0=\vep_0(H_1)>0$ and $\rho=\rho(H_1)>0$, such that for every $\vep\in (-\vep_0,\vep_0)\setminus\{0\}$ and every $\delta>0$, the Hamiltonian flow  admits a trajectory  whose action variables $I$  satisfy
\[\sup_{t>0}\|I(t)-I(0)\|\geq \rho,\quad \|I(0)-I_0\|<\delta+K\vep.\]

As for the neighborhood $V_{I_0}$ of $I_0$, 
taking $\delta$ and $\vep$ suitably small we can ensure that the initial condition $I(0)\in V_{I_0}$.
 This finishes the proof.
\end{proof}

Note that the Fr\'echet space $C^\omega=\cup_{\kappa>0}C^\omega_\kappa=\cup_{m\in \Z^+}C^\omega_{_{1/m}}$. Then by the definition of  Fr\'echet topology,
Corollary \ref{Main_Cor2}  follows directly from Theorem \ref{Main_Thm1}. 

\begin{proof}[Proof of Theorem \ref{Main_Thm2}]
	As we can see from the proof of Theorem \ref{H3a_genericity}, the perturbation functions  constructed by us depend only on $(q, t)\in \T^n\times\T$.
Meanwhile, in the proof of Theorem \ref{H3b_genericity}, the perturbation functions constructed by us depend only on $(q, \vph)\in \T^n\times\T^d$.
	This implies that the genericity of  assumptions {\bf(H3a)} and {\bf(H3b)} are established by constructing potential perturbations that are independent of $p$ and $I$. 
	  Therefore, assumptions {\bf(H3a)--(H3b) }are also open-dense in the $C_\kappa^\omega(\T^{n+d+1})$ space.
	 The remaining proof is just the same as that of Theorem \ref{Main_Thm1}.
\end{proof}

\begin{Rem} \label{rem:codimension}
We have verified that, in case that there is some degeneracy, it can 
be removed by adding some $\cos$ functions. 

Verifying that perturbations of this kind remove the degeneracy amounted 
to a determinant being non-zero. Clearly, if we modify the $\cos$ slightly, 
this condition will remain true.   This justifies the observation 
that the degeneracy can only fail of 
$H_1$ in a submanifold of infinite codimension.

We will not pursue this line of reasoning, but it seems that indeed, 
the set of directions transversal  to the manifold containing all the $H_1$ where 
diffusion fails is not only infinite dimensional, but also dense. 
This is indeed a very strong form of genericity.
\end{Rem}

\appendix
\section{Normally hyperbolic invariant manifolds}\label{appendix_Nor_hyp_theory}
Normally hyperbolic invariant manifold (NHIM) can be viewed as a natural generalization of  hyperbolic set.
The NHIM has not only stable and unstable directions, but also central directions (tangent to the manifold itself). The theory of normal hyperbolicity and the theory of partial hyperbolicity are closely related in their results and methods.  We refer the reader to the standard references \cite{Fen1971,Fen1977,HPS1977,Pesin2004}. 
In this appendix, we only review some classical results, including the existence of NHIMs, the existence of the 
stable and unstable manifolds and their invariant foliations, and the smoothness and the persistence of these manifolds. The definition of normal hyperbolicity that we adopt below is based on \cite{HPS1977}.

\subsection{The continuous case} 
Let $M$ be a smooth Riemannian manifold and $\Phi_t$ be an autonomous $C^{r_0}$ $( 1\leq r_0\leq\infty)$ flow defined on $M$. 
A $\Phi_t$-invariant submanifold  (probably with boundary) $N\subset M$ is called a \emph{normally hyperbolic invariant manifold} if for every $x\in N$ there is an invariant splitting 
\[T_xM=T_xN\oplus E_x^s\oplus E_x^u, \qquad\text{~with}\quad  D\Phi_t E_x^s=E^s_{\Phi_t(x)},\quad D\Phi_t E_x^u=E^u_{\Phi_t(x)},\quad \forall~t\in\R,\]
such that
\begin{align}\label{hyp splitting}
    v\in E^s_x  & \Longleftrightarrow C^{-1}e^{t\lambda_s }\|v\|\leq\|D\Phi_t(x)v\|\leq Ce^{t\mu_s }\|v\|,  \quad t\geq 0, \nonumber\\
    v\in E^u_x & \Longleftrightarrow  C^{-1}e^{t\mu_u }\|v\|\leq  \|D\Phi_t(x)v\|\leq Ce^{t\lambda_u }\|v\|,  \quad t\leq 0,\\ 
    v\in T_xN & \Longleftrightarrow C^{-1}e^{|t|\lambda_c }\|v\| \leq \|D\Phi_t(x)v\|\leq Ce^{|t|\mu_c }\|v\| ,  \quad t\in\R,\nonumber
\end{align}
where  the constant $C>1$, and the rates
\begin{equation*}
	\lambda_s\leq \mu_s<\lambda_c<0<\mu_c<\lambda_u\leq \mu_u.
\end{equation*}
The superscripts $c, u$ and $s$ stand for ``center", ``unstable" and ``stable", respectively.

\begin{Rem}
	For the Hamiltonian model considered in this paper, the rates $\lambda_c, \mu_c$ are  close to zero.
\end{Rem}

For a NHIM, the expansion and contraction in the central directions are weaker than those in the normal directions. Then the dynamics on $N$ is approximately neutral and the dynamics in the normal directions is hyperbolic. That is why we call it normally hyperbolic.

We point out that the manifold $M$ is not necessarily compact. As remarked in \cite{HPS1977, Bates_Lu_Zeng2000}, it suffices to assume that $\Phi_t$ is $C^{r_0}$ in a  neighborhood of $N$ with all the derivatives of order up to $r_0$ uniformly continuous and uniformly bounded.

\subsubsection{Stable and unstable manifolds}
Let us consider a small tubular neighborhood $U$ of the NHIM $N$. In both \cite{Fen1971} and \cite{HPS1977}
the existence of local stable and unstable manifolds in $U$, denoted by $W^{s,loc}_N$ and $W^{u,loc}_N$, are obtained by using the method of Hadamards's graph transform. In addition, $W^{s,loc}_N$ (resp. $W^{u,loc}_N$) consists of points for which all forward (resp. backward) iterates lie in $U$ and approach $N$. 
Then the (global) stable and unstable manifolds can be defined by , respectively, 
\[W^s_N=\bigcup_{t\leq 0}\Phi_t(W^{s,loc}_N), \qquad W^u_N=\bigcup_{t\geq 0}\Phi_t(W^{u,loc}_N).\]
The global stable and unstable manifolds can have a topological characterization:
\begin{equation}\label{definition of local stable}
  \begin{split}
     W^s_N =\{y\in M ~|~ \textup{dist}(\Phi_t(y),N)\to 0,\text{~as~} t\to+\infty\},\quad
         W^u_N=\{y\in M ~|~ \textup{dist}(\Phi_t(y),N)\to 0,\text{~as~} t\to-\infty\},
   \end{split}
\end{equation}
where ``dist$(\cdot,\cdot)$" is the distance induced by the Riemannian metric on $M$. In fact, the distance convergence in \eqref{definition of local stable} is exponential.

For each $x\in N$, we can also construct the stable and unstable manifolds with basepoint $x$: 
\begin{equation}\label{def_leave}
  \begin{split}
     W^s_x &=\{y\in M ~|~  \textup{dist}(\Phi_t(x),\Phi_t(y))\leq  \widetilde{C}_y\,e^{(\mu_s+\widetilde{\vep})t}, \text{~for~} t\geq 0\}, \\
     W^u_x &=\{y\in M ~|~  \textup{dist}(\Phi_t(x),\Phi_t(y))\leq  \widetilde{C}_y\,e^{(\lambda_u-\widetilde{\vep})t}, \text{~for~} t\leq 0 \}.
   \end{split}
\end{equation}
where  $\widetilde{C}_y$ is a constant, and $\widetilde{\vep}>0$ is any small number satisfying 
\[\mu_s+\widetilde{\vep}<\lambda_c,\qquad \lambda_u-\widetilde{\vep}>\mu_c.\] 
This tells us that the trajectories starting on $W^{s}_x$ or $W^u_x$ satisfy certain asymptotic growth rate conditions, and the growth rate shall be greater than that of the trajectories on $N$.

 It is important to realize that the manifolds $W^{s,u}_N$ are $\Phi_t$-invariant while  $W^{s,u}_x$ are not. Anyway, $W^{s}_N$ and $W^{u}_N$ can be foliated by the stable and unstable manifolds of points respectively, i.e., 
\begin{align}\label{foliated_by_leaves}
	W^{s}_N=\bigcup\limits_{x\in N} W^{s}_x, \qquad  W^{u}_N=\bigcup\limits_{x\in N} W^{u}_x.
\end{align}
For $x\neq x'$, one has $W^{s}_x\bigcap W^{s}_{x'}=\emptyset$ and $W^{u}_x\bigcap W^{u}_{x'}=\emptyset.$

 \subsubsection{Smoothness}\label{appendix_subsub_smooth}
The manifolds $W^{s,u}_x$ are as smooth as the flow $\Phi_t$, so $W^{s,u}_x$ are  $C^{r_0}$ $( 1\leq r_0\leq\infty)$ with  $T_xW^{s}_x=E_x^s$ and $T_xW^{u}_x=E_x^u$.  Nevertheless, the stable manifold $W^s_N$ and the unstable manifold $W^u_N$ have limited regularity even if the flow is $C^\infty$.
Their regularity is dictated by the ratio of the normal hyperbolicity and the central hyperbolicity. More precisely, we introduce the following integers (see \cite[Chapter 5]{Pesin2004} or \cite{HPS1977}):
\begin{align}\label{regularity_NHIM}
	\ell_u:=\max\left\{k=1,\cdots,r_0~:~k<\frac{\mu_s}{\lambda_c}\right\}, \quad 
	\ell_s:=\max\left\{k=1,\cdots,r_0~:~k<\frac{\lambda_u}{\mu_c}\right\}, \quad \ell:=\min\{\ell_u,\ell_s\}.
\end{align}
Clearly, $\ell_u$, $\ell_s$ and $\ell$ are all finite values even when $r_0=\infty$.

\begin{Pro}\label{pro_smooth_su}\cite{HPS1977}
	 The following properties hold:
 \begin{enumerate}
 	\item [\rm(I)] $W^{s}_N$ and $W^{u}_N$ are at least  $C^\ell$ differentiable. In fact,  $W^{s}_N$ is $C^{\ell_s}$ and $W^{u}_N$ is $C^{\ell_u}$. Then $N=W^s_N\cap W^u_N$ is $C^{\ell}$ differentiable.
 	\item [\rm(II)]  The stable foliation $\{W^s_x : x\in N\}$ is $C^\ell$ in the sense that $\bigcup_{x\in N}T^k_x W^s_x$ is a continuous bundle for each $1\leq k\leq \ell$, where $T^k$ denotes the $k$-th order tangent. Analogous result holds for the unstable foliation.
 \end{enumerate}
\end{Pro}

 The index $\ell$ means that $D\Phi_t$ expands $E^u$ (resp. contracts $E^s$) at rates at least $\ell$ times of its expansion (resp. contraction) rate in $TN$.  Thus,  such a manifold $N$ is also called \emph{$\ell$-normally hyperbolic}.

Property (II) implies that the invariant bundle $x\to E^{s,u}_x$ is $C^{\ell-1}$.
 In general, $W^{s,u}_N$ and $W^{s,u}_x$ are immersed manifolds and may fail to be embedded manifolds.  
 %$N$ and $W^{s,u}_N$  may fail to be $C^\infty$ even if the flow is analytic.

\subsection{The discrete case} 
For the Hamiltonian model considered in this paper, sometimes it is convenient to study the time-$T$ maps of the Hamiltonian flow. Hence, the study of NHIMs for maps is also needed.  
The definition of NHIM for diffeomorphisms is complete  analogous to the definition for flows. In fact, we can replace the continuous variable $t$  by a discrete variable without substantially changing any construction. Here, we only recall the definition. 

Let $f: M\to M$  be a $C^{r_0} ( 1\leq r_0\leq\infty)$ diffeomorphism, and $N\subset M$ be an invariant submanifold (probably with boundary). Assume that all the derivatives of order up to $r$ of $f$ are uniformly continuous and uniformly bounded in a neighborhood of $\Lambda$. 
 Then $N$ is called a \emph{normally hyperbolic invariant manifold} for $f$ if for every $x\in N$ there is an invariant splitting
$T_xM=T_xN\oplus E_x^s\oplus E_x^u,$ with
$Df E_x^s=E^s_{f(x)}$ and $Df E_x^u=E^u_{f(x)}$,
such that
\begin{equation}
\begin{split}
    v\in E^s_x  & \Longleftrightarrow C^{-1}\alpha^k_s \|v\|\leq\|Df^k(x)\,v\|\leq C\beta^k_s \|v\|,  \quad k\geq 0,\\
    v\in E^u_x & \Longleftrightarrow   C^{-1}\beta^k_u \|v\|\leq\|Df^k(x)\,v\|\leq C\alpha^k_u \|v\|, \quad k\leq 0,\\ 
    v\in T_xN & \Longleftrightarrow C^{-1}\alpha^{|k|}_c \|v\|\leq\|Df^k(x)\,v\|\leq C\beta^{|k|}_c \|v\|,  \quad k\in \Z,
\end{split}
\end{equation}
where  the constant $C>1$, and the rates
$
	0<\alpha_s\leq \beta_s<\alpha_c<1<\beta_c<\alpha_u\leq \beta_u.
$

\subsection{Persistence and dependence on parameters}\label{subsubsec_per_NHIM}
For applications, it is important to study the persistence of the NHIM  under perturbations, and if the persistent manifold depends smoothly on the perturbation parameter.

\begin{The} \cite{Fen1971,HPS1977,Bates_Lu_Zeng2000}\label{appendix_persistence}
	Let $N\subset M$ be a submanifold without boundary and  $N$ is a $\ell$-normally hyperbolic invariant manifold ($\ell$ is defined in \eqref{regularity_NHIM}) for the $C^{r_0}$ flow $\Phi_t$ generated by the vector field $X$. Then for the vector field $Y$ which is $C^1$-close to $X$, there exists a unique normally hyperbolic and  $\Phi^Y_t$-invarnat manifold  $N_Y$, which, is $C^\ell$ diffeomorphic and close to $N$. In particular, $N_Y$ is also $\ell$-normally hyperbolic.
	The local stable manifold $W^{s,loc}_{N_Y}$ and local unstable manifold $W^{u, loc}_{N_Y}$ are $C^\ell$ close to those of $N$. 
\end{The}

\begin{Rem}
The $C^1$-closeness between the vectors is enough because the change in the rates $\lambda_{\iota}, \mu_\iota$, $\iota=s,c,u$, can be controlled by the $C^1$ distance. The persistent manifold $N_Y$ is still  $C^\ell$ for the reason that the exponents in \eqref{hyp splitting}  for $N_Y$ is very close to those of $N$, and hence the index $\ell$ in \eqref{regularity_NHIM} remains unchanged.
\end{Rem}

For the case where the submanifold $N$ has non-empty boundary,  a locally invariant and normally hyperbolic manifold persists. To prove it, one can  construct a slightly modified system for which the normally hyperbolic manifolds are globally invariant under the modified flow. Then, by Theorem \ref{appendix_persistence} there is a unique persistent NHIM. This invariant manifold for the slightly modified system would be a locally invariant manifold for the original system. It also implies the locally invariant manifold is not unique in general. Nevertheless, there are some  cases for which the uniqueness holds. For example, the locally invariant manifold has KAM tori bounding them.

\subsubsection{Smooth parameter dependence}\label{subsubsection_NHIM_Parameter}
The persistent manifold depends smoothly on the perturbation parameter. In fact, this can be achieved by considering a extended system.

More precisely, let $\Phi_{t,\vep}: M\to M$ be a family of $C^{r_0}$ flows depending smoothly on the parameter $\vep\in [0,1]$, and assume $N_0$ is a  NHIM of the flow $\Phi_{t,0}$ satisfying \eqref{regularity_NHIM}, which  means $N_0$ is $C^\ell$ differentiable. Then, we can define an extended space $\widehat{M}=M\times [0,1]$ and an extended manifold $\widehat{N}_0=N_0\times [0,1]$. To define the extended flow, 
we introduce an external scaling parameter $\alpha\geq 0$, and construct a $C^{r_0}$ flow 
$\widehat{\Psi}^{\alpha}_t$ on the manifold $\widehat{M}$ by
\[\widehat{\Psi}^{\alpha}_t\big(x, \vep\big):=\big(\Phi_{t,\alpha\vep}(x),\vep\big)\quad \text{for all~}(x,\vep)\in M\times[0,1]\]
For $\alpha=0$,  $\widehat{N}_0$ is a $C^\ell$ NHIM of the flow $\widehat{\Psi}^0_t=\Phi_{t,0}\times \mathrm{Id}$. 
When $\alpha$ is fixed and sufficiently small, the flow $\widehat{\Psi}^{\alpha}_t$ is close enough to $\Phi_{t,0}$, then the persistence result  implies that the flow $\widehat{\Psi}^{\alpha}_t$ has a  NHIM $\widehat{N}_\alpha\subset\widehat{M}$, and $\widehat{N}_\alpha$ is $C^\ell$ close and diffeomorphic to $\widehat{N}_0$. Note that $\widehat{N}_\alpha$  can be decomposed into
\[\widehat{N}_\alpha=\coprod_{\vep\in[0,1]} N_{\alpha\vep}\times\{\vep\},
\]
where each $N_{\alpha\vep}$ is invariant under $\Phi_{t,\alpha\vep}$ and
 depends $C^\ell$-smoothly on the parameter $\vep\in[0,1]$. Also, it is not difficult to check that each $N_{\alpha\vep}$ is a NHIM.
  Therefore, we conclude that for $\vep\in[0,\alpha]$,  $N_{\vep}$ is a normally hyperbolic and $\Phi_{t,\vep}$-invariant manifold depending $C^\ell$-smoothly on the parameter $\vep$.
  
  %Analogous results hold for diffeomorphisms. 

\section{The scattering map}\label{appendix_scattering_map}
The scattering map is used to  describe homoclinic excursions. This map is introduced explicitly in \cite{DLS2000} and  enjoys remarkable geometric properties \cite{DLS2008}.

Recall that   $W^s_N$ and $W^u_N$ are, respectively, foliated by the stable leaves $W_x^s$ and the unstable leaves $W^u_x$, $x\in N$. For any point $x\in W^s_N$ (resp. $x\in W^u_N$), there is a unique point $x_+\in N$ (resp. $x_-\in N$) such that 
$x\in W^s_{x_+}$ (resp. $x\in W^u_{x_-}$). Then we can define the
 \emph{wave maps} which are projections along the  the leaves: 
\begin{equation}\label{appendix_wavemaps}
\Omega_{+}: W^{s}_N\longrightarrow N, \quad x\longmapsto x_{+};\quad\quad
     \Omega_{-}: W^{u}_N\longrightarrow N, \quad x\longmapsto x_{-}.
\end{equation}
The wave maps are  $C^\ell$ smooth as a result of  the $C^\ell$-foliation property (see Proposition \ref{pro_smooth_su}). 

To define the scattering map, we need the following transversality conditions:
\begin{enumerate}
	\item $W^s_N$ and $W^u_N$ have a transversal intersection along a homoclinic manifold $\Gamma$, i.e., for each $z\in\Gamma$,
\begin{equation}\label{tran_intersection_condition_1}
	T_zM=T_z W^u_N+ T_zW^s_N, \quad 
T_z\Gamma=T_z W^u_N\cap T_zW^s_N.
\end{equation}
\item  $\Gamma$ is transverse to the foliations of the stable/unstable
manifolds at each point $z\in\Gamma$, that is
\begin{equation}\label{tran_intersection_condition_2}
	T_z W^s_N= T_z\Gamma\oplus T_z W^s_{x_+}, \quad T_z W^u_N= T_z\Gamma\oplus T_z W^u_{x_-},
\end{equation}
where $x_\pm$ are the uniquely defined points in $N$ satisfying $z\in W^{s}_{x_+}\cap W^u_{x_-}$.
\end{enumerate} 
%Condition \eqref{tran_intersection_condition_2} can be  implied from $T_z M$ $=$ $T_z\Gamma\oplus T_z W^s_{x_+}\oplus T_z W^u_{x_-}$. 
% 
%This condition is used to to define the scattering map locally, and  this local definition could have some monodromy if extended to domains with non-contractible loops.

 \begin{Rem}\label{existence_homo_mfld}
 Obviously, $\text{dim}(\Gamma)=\text{dim}(N)$.
If \eqref{tran_intersection_condition_1}--\eqref{tran_intersection_condition_2} hold at some point $z^*\in  W^u_N\cap W^s_N$, then by the implicit function theorem the transversality conditions are also satisfied for all $z\in W^u_N\cap W^s_N$ close to $z^*$. Hence we can find a locally unique manifold $\Gamma\ni z^*$  satisfying \eqref{tran_intersection_condition_1}--\eqref{tran_intersection_condition_2}. In addition, $\Gamma$  is $C^\ell$. 
 \end{Rem}

Consider the wave maps restricted to the $C^\ell$ manifold $\Gamma$, denoted by
\begin{align*}
\Omega_{+}|_{_{\Gamma}}: \Gamma\to \Omega_{+}(\Gamma) \quad\text{~and~}\quad \Omega_{-}|_{_{\Gamma}}: \Gamma\to \Omega_{-}(\Gamma).
\end{align*}
They are $C^\ell$ local diffeomorphisms in general. 
Even if  $\Omega_{\pm}|_{_{\Gamma}}$ are locally invertible, they could fail to be invertible in a domain with non-contractible loops, see \cite{DLS2000, DLS2006orbits} for more examples.

Following \cite{DLS2008}, we say $\Gamma\subset  W^u_N\cap W^s_N$  a \emph{homoclinic channel} if it satisfies  \eqref{tran_intersection_condition_1}--\eqref{tran_intersection_condition_2} and $\Omega_{\pm}|_{_{\Gamma}}$ are $C^\ell$ diffeomorphisms. Then, the \emph{scattering map} $\sigma^\Gamma$ associated to the homoclinic channel $\Gamma$ is 
 \begin{equation}
	\sigma^\Gamma=\Omega_+|_{_{\Gamma}}\circ\big(\Omega_-|_{_{\Gamma}}\big)^{-1}~:~ \Omega_{-}(\Gamma)\longrightarrow \Omega_{+}(\Gamma).
 \end{equation}
Note that $\sigma^\Gamma$ is $C^\ell$ smooth. Clearly,
the definition of  scattering map depends on the homoclinic channel. Sometimes we will omit $\Gamma$ from the notation when there is no confusion.
  One can also expect to have infinitely many $\Gamma$, and each of which has a different scattering map. 
\begin{Rem}
 We shall note that there is no actual orbit of the flow $\Phi_t$ starting from $x_-$ to $x_+$. If $x_+=\sigma^\Gamma(x_-)$, then we infer from \eqref{def_leave} that there is a point $z\in \Gamma$ satisfying
\[
\textup{dist}\big(\Phi_t(z),\Phi_t(x_{+})\big) \leq \widetilde{C}e^{(\mu_s+\widetilde{\vep})t}, \quad t\to +\infty; \quad
\textup{dist}\big(\Phi_t(z),\Phi_t(x_{-})\big) \leq \widetilde{C}e^{(\lambda_u-\widetilde{\vep})t}, \quad t\to -\infty.\]
 \end{Rem}

\section{Proofs of Theorem \ref{Thm_scattering_shadowing} and Theorem \ref{existence_orbits}} \label{Appendix_section}

For the reader's convenience we repeat the relevant material from \cite{Gidea_Delallave_seara2020} to
 give a sketch of the proof of Theorems \ref{Thm_scattering_shadowing}--\ref{existence_orbits}. 
 
\begin{proof}[Proof of Theorem \ref{Thm_scattering_shadowing}]
Let $N=[1/\vep]$ be the integer part of  $1/\vep$ and consider the following orbit of $s_\vep$:
 \begin{equation}
	x_i=s_\vep(x_{i-1}),\quad x_0=(I_0,\theta_0),\quad i=1,\cdots,N.
\end{equation}
By assumption we have a curve $\gamma:[0,1]\to\Lambda_0$ with $\gamma(0)=x_0$ which is a solution to 
\begin{align}\label{Appen_HF}
	\dot{x}=\mathbf{J}\nabla \cL^*(x),\quad x\in\Lambda_0,
\end{align} 
and the set $V \subset \text{Dom}(\cL^*)$ is a neighborhood of $\gamma([0,1])$. 

We claim that all points of the sequence $\{x_i\}_{i=0}^N$  lies in $V$ as long as $\vep$ is small enough.
Let us take a sequence $y_{i}:=\gamma(i\vep)$ with $i=0,\cdots,N$. Denoting  by $\phi_t$ the Hamiltonian flow associated to the equation 
\eqref{Appen_HF}, we  invoke the Gronwall inequality to \eqref{Appen_HF} to find a constant $C_1>0$ such that
\begin{equation}\label{Appen_est_1}
	\|\phi_\vep(x)-\phi_\vep(x')\|\leq e^{C_1\vep}\|x-x'\|, \quad x, x'\in\Lambda_0.
\end{equation}
Recalling the scattering map $s_\vep$ is, up to $O(\vep^2)$, the $\vep$-flow of the  Hamiltonian $\cL^*(I,\theta)$, we obtain
\begin{equation}\label{Appen_est_2}
	\|s_\vep(x)-\phi_\vep(x)\|\leq C_2\vep^2,\quad x\in\Lambda_0,
\end{equation}
where $C_2>0$ is a constant. We also remark that $C_1$ and $C_2$ depend on $H_1$.

As $y_{i}=\gamma(i\vep)=\phi_\vep(y_{i-1})$ for each $i=1,\dots, N$,  it follows from  inequalities \eqref{Appen_est_1}--\eqref{Appen_est_2} that
\begin{align*}
\|x_i-y_i\|=\|s_\vep(x_{i-1})-\phi_\vep(y_{i-1})\|&\leq \|s_\vep(x_{i-1})-\phi_\vep(x_{i-1})\|+\|\phi_\vep(x_{i-1})-\phi_\vep(y_{i-1})\|\\
&\leq C_2\vep^2+e^{C_1\vep}\|x_{i-1}-y_{i-1}\|.	
\end{align*}
Denoting  $w_i:=\|x_i-y_i\|+\frac{C_2\vep^2}{e^{C_1\vep}-1}$, the  inequality  above implies that
\[w_i\leq e^{C_1\vep}w_{i-1}, \quad \text{for~} i=1,\cdots, N\]
Since $x_0=y_0$,  we get $w_0=\frac{C_2\vep^2}{e^{C_1\vep}-1}$ and
\[
w_i\leq w_0 \left(e^{C_1\vep }\right)^i.
\]
Consequently, 
\begin{align}\label{xiyi_est}
	\|x_i-y_i\|\leq \frac{C_2\vep^2}{e^{C_1\vep}-1}\bigg[\left(e^{C_1\vep}\right)^i-1\bigg]\leq \frac{C_2\vep }{C_1}e^{C_1}, \quad  \text{for~} i=1,\cdots, N,
\end{align}
where the last inequality follows from $e^{C_1\vep}-1\geq C_1\vep$
 and $N=[1/\vep]$. Denoting $K:=C_2e^{C_1}/C_1$, inequality \eqref{xiyi_est} becomes \[\text{dist}(x_i,\gamma(i\vep))\leq K\vep.\] Finally, as
  the distance $d:=\text{dist}\big(\gamma([0,1]), \partial V\big)>0$,  there exists a small $\vep_1=\vep_1(H_1)\in(0, d/K)$ such that for each $0<|\vep|<\vep_1$, the sequence $\{x_i\}_{i=0}^N$ lies in the set $V$. This proves our claim.
 
As $\gamma_\vep=k_\vep\circ \gamma\subset \Lambda_\vep$, we set $x^\vep_i=k_\vep\circ x_i\in  \Lambda_\vep$. By enlarging $K$ if necessary, we have $\text{dist}(x^\vep_i,\gamma_\vep(i\vep))\leq K \vep$, and the sequence $\{x^\vep_i\}_{i=0}^N$ lies in the neighborhood  $V_\vep:=k_\vep\circ V$ of the curve $\gamma_\vep$. Since almost every point in $V_\vep$ is recurrent for $\whf_\vep\big|_{\La_0}$,  our theorem follows 
immediately from the shadowing lemma \ref{Appen_lemma} below.
\end{proof}

\begin{Lem}[Shadowing Lemma for pseudo-orbits]\label{Appen_lemma}
For a diffeomorphism $f: M\to M$, we let $\Lambda$ be a NHIM and $\sigma$ be a scattering map. Assume that $f$ preserves a measure absolutely continuous with respect to the Lebesgue measure on $\Lambda$, and that $\sigma$ sends positive measure sets to positive measure sets.
	
	Let $x_i=\sigma(x_{i-1})$, $i=1,\cdots, N$ be a finite orbit of the scattering map, which is contained in an open set $V$ with almost every point of $V$ recurrent for $f$. Then for every $\delta>0$, there is an orbit $\{z_i\}_{i=0}^N$ of $f$ satisfying $z_{i+1}=f^{k_i}(z_i)$ for some integer $k_i>0$, and $\text{dist}(z_i,x_i)<\delta$ for all $i=0,\cdots, N$.	
\end{Lem}
See \cite[Theorem 3.6]{Gidea_Delallave_seara2020} for a complete proof of the lemma above. 

\begin{proof}[Proof of Theorem \ref{existence_orbits}]
	By assumption \textbf{(H3b)}, $\mathbf{J}\nabla\cL^*(I,\theta)$ is transverse to the level set $\{I=I_*\}$ at some point $(I_*,\theta_*)\in \text{Dom}(\cL^*)$ $\subset \Lambda_0$, which implies that there exist two closed balls 
	\[D_I=\{I\in\R^d ~:~\|I-I_*\|\leq r_1\},\quad D_\theta=\{\theta\in\T^d ~:~\|\theta-\theta_*\|\leq r_1 \},\]
 such that $\mathbf{J}\nabla\cL^*(I,\theta)$ is still transverse to the level set $\{I=I'\}$  at each point $(I',\theta')\in$ $D_I\times D_\theta$ $\subset\text{Dom}(\cL^*)$. Here, we stress that the size of radius $r_1$  does not depend on $\vep$. 

Let $\gamma :$  $ [0,1]\to \Lambda_0$  be an integral curve, starting from $(I_*,\theta_*)$, induced by the  vector field $\mathbf{J}\nabla \cL^*$, then there is a time $t_*\in(0,1)$ independent of $\vep$, such that $\gamma([0,t_*])\subset D_I\times D_\theta$.   It follows  that the curve $\gamma(t)=\big(I(\gamma(t)),\theta(\gamma(t))\big)$ is transverse to every level set $\{I=I(\gamma(t))\}$  for each  $t\in [0,t_*]$. Thus,
\begin{align}\label{Igat_0}
	\|I(\gamma(t_*))-I(\gamma(0))\|=\|I(\gamma(t_*))-I_*\|\geq 2\rho,
\end{align}
for some constant $\rho>0$ independent of $\vep$.

Recalling that in Theorem \ref{Thm_scattering_shadowing}  we use $f_\vep$ to denote the time-$2\pi$ map for the flow of the Hamiltonian  $\HH_\vep$, and  use  $\whf_\vep\big|_{\La_0}= (k_\vep)^{-1}\circ f_\vep\big|_{\La_\vep}\circ k_\vep$ to denote the parameterized  map defined on $\Lambda_0$. Denoting  
\[
V:=\bigcup_{k\geq 0} \widehat{f}_\vep^k(D_I\times D_\theta),
\]
it is a $\widehat{f}_\vep$-invariant set in $\Lambda_0$, i.e. $\widehat{f}_\vep(V)\subset V$. The measure of $V$ is either finite or $+\infty$. 

\textbf{Case 1:} The measure of $V$ is finite.  By the $\widehat{f}_\vep$-invariance we infer that almost every point in $V$ is recurrent. As $V$ is a neighborhood of the curve $\gamma|_{[0,t_*]}$,  we can apply Theorem \ref{Thm_scattering_shadowing}. Indeed, let  $\vep_1=\vep_1(H_1)>0$ and $K>0$ be given in Theorem \ref{Thm_scattering_shadowing},  
 then for every small $\delta>0$ and every $\vep\in(-\vep_1,\vep_1)\setminus\{0\}$, there exists
an orbit $\{z_i\}_{i=0,\cdots,m}$  of the diffeomorphism $f_\vep$   
 and  for each $i=0,\cdots,m-1$,
\[
z_{i+1}=f^{k_i}_\vep(z_i),\quad \text{~for some~} k_i\in\Z^+;\qquad d(z_i,\gamma_\vep(t_i))<\delta+K\vep 
\]
where $0=t_0<t_1<\cdots<t_m=t_*$ and $\gamma_\vep=k_\vep\circ\gamma\subset\Lambda_\vep$. In particular, $d(z_0,\gamma_\vep(0))<\delta+K\vep$ and $d(z_m,\gamma_\vep(t_*))$ $<\delta+K\vep$. Using \eqref{Igat_0} we obtain   
\[\|I(z_m)-I(z_0)\|>2\rho-2(\delta+K\vep).\]
 Suppose $\vep<\rho/(4K)$ if necessary,  we can choose $\delta<\rho/4$ such that 
 \[\|I(z_m)-I(z_0)\|>\rho.\]
 Moreover, $\|I(z_0)-I_*\|<\delta+K\vep$. 

\textbf{Case 2:} The measure of $V$ is $+\infty$. Then the existence of diffusing orbits is evident. Indeed, it implies that for every  $\rho>0$, there is an orbit $\{ \hat{z}_i : i=0,\cdots,m\}\subset \Lambda_0$  of the map $\widehat{f}_\vep$, such that $\hat{z}_0\in D_I\times D_\theta$ and
$\|I(\hat{z}_m)-I(\hat{z}_0)\|>2\rho$.  Thus, taking $z_i=k_\vep\circ\hat{z}_i$ for each $i$, the sequence $\{z_i\}_{i=0,\cdots,m}$ is exactly an orbit of  $f_\vep$, and
$\|I(z_m)-I(z_0)\|>2\rho-O(\vep)>\rho$. Finally, to ensure $I(z_0)$ is $(\delta+K\vep)$-close to $I_*$, it suffices to assume that the radius $r_1$ of $D_I$  satisfies $r_1\leq \delta/4$.

This completes the proof.
\end{proof}

\noindent{\bf Acknowledgments}
Qinbo Chen thanks the Georgia Institute of Technology for its warm hospitality during his visit.  
 He also wishes to thank KTH Royal Institute of Technology where the final version of this manuscript was completed.
Rafael de la Llave was partially supported by DMS 1800241.

%\bibliographystyle{plain}
%\bibliography{mybib}
\def\cprime{$'$}

\end{document}